\RequirePackage{fix-cm}
\documentclass[12pt]{amsart}
\usepackage{amsmath}
\usepackage{setspace}
\usepackage{fancyhdr, lipsum}
\usepackage{mathrsfs}
\usepackage{xifthen}

\usepackage{geometry}
\usepackage{color}
\usepackage{textcomp}
\usepackage{amsthm}
\usepackage{amstext}
\usepackage{amsfonts}
\usepackage{amssymb}
\usepackage{fixltx2e}
\usepackage{graphicx}

\usepackage{enumitem}	

\usepackage{tikz}
\usepackage{tikz-cd}
\usepackage{caption}
\usetikzlibrary{calc,backgrounds}
\usetikzlibrary{matrix,arrows,decorations.pathmorphing}

\usepackage{bigstrut}

\DeclareMathAlphabet{\mathpzc}{OT1}{pzc}{m}{it}

\newcommand{\Spec}{\text{Spec}}
\newcommand{\Spf}{\text{Spf }}
\newcommand{\Hom}{\text{Hom}}

\newcommand{\Lie}{\text{Lie}}
\newcommand{\GL}{\text{GL}}

\newcommand{\Ext}{\text{Ext}}

\newcommand{\Fil}{\text{Fil}}
\newcommand{\Isom}{\textbf{Isom}}

\newcommand{\Res}{\text{Res}}
\newcommand{\gr}{\text{gr}}
\newcommand{\rig}{\text{rig}}
\newcommand{\univ}{\text{univ}}

\newcommand{\et}{\text{\'et}}
\newcommand{\smooth}{\text{sm}}

\newcommand{\hooklongrightarrow}{\lhook\joinrel\longrightarrow}

\newcommand{\Z}{\mathbb{Z}}
\newcommand{\Q}{\mathbb{Q}}
\newcommand{\C}{\mathbb{C}}
\newcommand{\F}{\mathbb{F}}
\newcommand{\Gm}{\mathbb{G}_m}

\newcommand{\inv}{^{-1}}

\newcommand{\D}{\mathbb{D}}
\newcommand{\module}{\Lambda}
\newcommand{\genring}{R}
\newcommand{\genscheme}{S}

\newcommand{\sheafof}[1]{\mathcal{O}_{#1}}
\newcommand{\isommatchingtensors}{\mathcal{P}}
\newcommand{\vectorbundle}{\mathscr{E}}
\newcommand{\dieudonnemodule}{M}

\newcommand{\completion}{\widehat}
\newcommand{\genqisog}{\iota}

\newcommand{\levi}{L}
\newcommand{\parabolic}{P}
\newcommand{\unipotent}{U}

\newcommand{\unram}{\text{un}}

\newcommand{\completemaxunram}{\breve{\Q}_p}
\newcommand{\completemaxunramint}{\breve{\Z}_p}
\newcommand{\BT}{X}
\newcommand{\deform}{\mathscr}

\newcommand{\lift}{\mathcal}
\newcommand{\liftlow}{\textbf}
\newcommand{\liftother}{\hat}

\newcommand{\DeligneLusztig}[3]{X_{\{#2\}}^{#1}([#3])}
\newcommand{\Def}{\text{Def}}
\newcommand{\Nilp}{\text{Nilp}}
\newcommand{\RZ}{\text{RZ}}
\newcommand{\weildescent}{\Phi}
\newcommand{\weilgroup}[1]{W_{#1}}
\newcommand{\analtower}[1]{#1^\infty}

\newcommand{\gencocharacter}{\lambda}
\newcommand{\newtonmap}{\nu}

\newcommand{\unramext}{E}
\newcommand{\integerring}{\mathscr{O}}
\newcommand{\EL}{\widetilde}
\newcommand{\ELgroup}[1][]{%
\ifthenelse{\isempty{#1}}{{\text{Res}_{\unramext|\mathbb{Q}_p} \GL_n}}{\text{Res}_{\unramext_{#1}|\mathbb{Q}_p} \GL_{n_{#1}}}%
}
\newcommand{\ELgroupintegral}[1][]{%
\ifthenelse{\isempty{#1}}{{\text{Res}_{\integerring|\mathbb{Z}_p} \GL_n}}{\text{Res}_{\integerring|\mathbb{Z}_p} \GL_{#1}}%
}

\newcommand{\generalhodge}{\underline}

\newcommand{\localreflexfield}{E}

\newcommand{\levelatp}{{K_p}}

\newcommand{\tatemodule}[1]{T_p(#1)}

\makeatletter

\theoremstyle{theorem}
\newtheorem{theorem}[subsubsection]{Theorem}
\newtheorem{lemma}[subsubsection]{Lemma}
\newtheorem{prop}[subsubsection]{Proposition}

\newtheorem{conj}[subsubsection]{Conjecture}

\newcounter{introtheorem}
\newtheorem*{theorem*}{Theorem}
\newtheorem*{lemma*}{Lemma}

\theoremstyle{definition}
\newtheorem{example}[subsubsection]{Example}

\newenvironment{remark}[1][Remark.]{\begin{trivlist}\item[\hskip \labelsep {\bfseries #1}]}{\end{trivlist}}

\numberwithin{equation}{subsubsection}

\makeatother

\begin{document}

\tikzset{
    node style sp/.style={draw,circle,minimum size=\myunit},
    node style ge/.style={circle,minimum size=\myunit},
    arrow style mul/.style={draw,sloped,midway,fill=white},
    arrow style plus/.style={midway,sloped,fill=white},
}

\title[Harris-Viehmann conjecture for HN-reducible Rapoport-Zink spaces]{Harris-Viehmann conjecture for Hodge-Newton reducible Rapoport-Zink spaces}
\author{Serin Hong}
\address{Department of Mathematics, California Institute of Technology}
\email{shong2@caltech.edu}

\maketitle

\rhead{}


\chead{}


\begin{abstract}
Rapoport-Zink spaces, or more generally local Shimura varieties, are expected to provide geometric realization of the local Langlands correspondence via their $l$-adic cohomology. Along this line is a conjecture by Harris and Viehmann, which roughly says that when the underlying local Shimura datum is not basic, the $l$-adic cohomology of the local Shimura variety is parabolically induced. 

We verify this conjecture for Rapoport-Zink spaces which are Hodge type and Hodge-Newton reducible. The main strategy is to embed such a Rapoport-Zink space into an appropriate space of EL type, for which the conjecture is already known to hold by the work of Mantovan. 
\end{abstract}

\tableofcontents

\section{Introduction}
In \cite{Rapoport-Zink96}, Rapoport and Zink constructed formal moduli spaces of $p$-divisible groups which give rise to local analogues of PEL type Shimura varieties. These spaces, which are now called \emph{Rapoport-Zink spaces}, have played a crucial role in the study of the local Langlands correspondence. Perhaps the most striking example is Harris and Taylor's proof in \cite{Harris-Taylor01} of the local Langlands conjecture for $\GL_n$. A key point of their proof is that one can realize the local Langlands correspondence in the $l$-adic cohomology of Lubin-Tate spaces, which are Rapoport-Zink spaces that parametrize $p$-divisible groups of dimension $1$.

The theory of Rapoport-Zink spaces suggests that it should be possible to realize many cases of the local Langlands correspondence via the $l$-adic cohomology of some local analogues of Shimura varieties. Motivated by this, Rapoport and Viehmann in \cite{Rapoport-Viehmann14} formulated the idea that there should exist a general theory of local analogues of Shimura varieties, which they called \emph{local Shimura varieties}. They described a conjectural form of this theory which starts with a group theoretic datum called a \emph{local Shimura datum} and associates to this datum a tower of analytic spaces which enjoys analogous properties to the properties of Shimura varieties. Shortly after this conjectural formulation, Scholze in his Berkeley lectures \cite{Scholze14} gave a construction of local Shimura varieties in some perfectoid category which he called the category of \emph{diamonds}. 
For local Shimura data that arise from Shimura varieties of Hodge type, W. Kim in \cite{Kim13} constructed Hodge type Rapoport-Zink spaces which may serve as integral models of the corresponding local Shimura varieties.

There are two key cohomology conjectures, namely the Kottwitz conjecture and the Harris-Viehmann conjecture, which predicts how the $l$-adic cohomology of local Shimura varieties should realize the local Langlands correspondence. The Kottwitz conjecture, originally formulated by Kottwitz and introduced by Rapoport in \cite{Rapoport94}, concerns realization of supercuspidal representations when the underlying local Shimura datum is \emph{basic}. The Harris-Viehmann conjecture, originally formulated by Harris in \cite{Harris00} and later modified by Viehmann, gives an inductive formula for the cohomology when the underlying local Shimura datum is not basic.

The primary purpose of this paper is to prove the Harris-Viehmann conjecture for certain local Shimura data that arise from Shimura varieties of Hodge type. We will work in the setting of Rapoport-Zink spaces, as our proof will use previously known results for Rapoport-Zink spaces of PEL type. However, our argument should work as well in the setting of local Shimura varieties constructed by Scholze.

Let us now introduce the notations and terminologies necessary for a precise statement of our result. We fix a prime $p>2$, 
and set up some standard notations as follows: we write $\overline{\F}_p$ and $\overline{\Q}_p$ respectively for a fixed algebraic closure of $\F_p$ and $\Q_p$; $\Q_p^\unram$ for the maximal unramified extension of $\Q_p$; $\C_p$ and $\completemaxunram$ respectively for the $p$-adic completion of $\overline{\Q}_p$ and $\Q_p^\unram$; and $\completemaxunramint$ for the ring of integers of $\completemaxunram$. We also fix an \emph{unramified local Shimura datum of Hodge type}, which is a tuple $(G, [b], \{\mu\})$ consisting of a connected reductive group $G$ over $\Z_p$, a $\sigma$-conjugacy class $[b]$ of elements in $G(\check{\Q}_p)$, and a $G(\completemaxunramint)$-conjugacy class $\{\mu\}$ of cocharacters of $G$ satisfying certain axioms (see \ref{definition of unram local Shimura data of Hodge type} for details). Let $\localreflexfield$ denote the field of definition of $\{\mu\}$, which is an unramified finite extension of $\Q_p$. With a suitable choice of $b \in [b]$, the datum $(G, [b], \{\mu\})$ gives rise to a $p$-divisible group $\BT$ over $\overline{\F}_p$ with some additional structures induced by the group $G$. Let $J_b$ be an algebraic group over $\Q_p$ with functor of points
\[ J_b(R) = \{ g \in G(\genring \otimes_{\Q_p} \completemaxunram) : gb\sigma(g)\inv = b\}\]
for any $\Q_p$-algebra $\genring$.

To the pair $(G, b)$, we associate a Rapoport-Zink space of Hodge type $\RZ_{G, b}$ following W. Kim's construction in \cite{Kim13}. The space $\RZ_{G, b}$ is a formal scheme over $\Spf (\completemaxunramint)$ whose isomorphism class depends only on the datum $(G, [b], \{\mu\})$. It also has a rigid analytic generic fiber $\RZ_{G, b}^\rig$ which is equipped with a tower of \'etale covers $\analtower{\RZ_{G, b}} := \{\RZ_{G, b}^\levelatp\}$ where $\levelatp$ runs over open compact subgroups of $G(\Z_p)$. 
The $l$-adic cohomology groups
\[ H^i(\RZ_{G, b}^\levelatp) := H^i_c (\RZ_{G, b}^\levelatp \otimes_{\completemaxunram} \C_p, \Q_l( \dim \RZ_{G, b}^\levelatp)) \quad \text{ for } i>0\]
fit into a tower $\{ H^i(\RZ_{G, b}^\levelatp)\}$ with a natural action of $G(\Q_p) \times \weilgroup{\localreflexfield} \times J_b(\Q_p)$ where $\weilgroup{\localreflexfield}$ is the Weil group of $\localreflexfield$. For an $l$-adic admissible representation $\rho$ of $J_b(\Q_p)$, we define a virtual representation of $G(\Q_p) \times \weilgroup{\localreflexfield}$
\[ H^\bullet(\RZ^\infty_{G, b})_\rho := \sum_{i, j \geq 0} (-1)^{i+j} \varinjlim_{\levelatp} \Ext^j_{J_b(\Q_p)} (H^i(\RZ_{G, b}^\levelatp), \rho). \]

We prove the Harris-Viehmann conjecture under the assumption that the datum $(G, [b], \{\mu\})$ is \emph{Hodge-Newton reducible}. Roughly speaking, this means that the datum $(G, [b], \{\mu\})$ naturally reduces to a local Shimura datum for some Levi subgroup $\levi$ of $G$. More precisely, there exists a choice of $b \in [b] \cap \levi(\completemaxunram)$ and $\mu \in \{\mu\}$ which factors through $\levi$ such that the tuple $(\levi, [b], \{\mu\})$ is an unramified local Shimura datum of Hodge type (see \ref{definition of hodge-newton reducibility} for details). 

Now we can state our main result as follows:
\begin{theorem*}
Assume that the unramified local Shimura datum of Hodge type $(G, [b], \{\mu\})$ is Hodge-Newton reducible with respect to a parabolic subgroup $\parabolic$ of $G$ with Levi factor $\levi$. Choose $b \in [b] \cap \levi(\completemaxunram)$ which gives rise to a $p$-divisible group over $\overline{\F}_p$ with additional structures induced by $\levi$. For any admissible $\overline{\Q}_l$-representation $\rho$ of $J(\Q_p)$, we have the following equality of virtual representations of $G(\Q_p) \times \weilgroup{\localreflexfield}$:
\[H^\bullet(\RZ^\infty_{G, b})_\rho = \text{Ind}_{\parabolic(\Q_p)}^{G(\Q_p)} H^\bullet(\RZ^\infty_{\levi, b})_\rho.\]
In particular, the virtual representation $H^\bullet(\RZ^\infty_{G, b})_\rho$ contains no supercuspidal representations of $G(\Q_p)$.
\end{theorem*}

Let us record some previously known results on the Harris-Viehmann conjecture. The earliest result of this form is Boyer's work in \cite{Boyer99} for Drinfeld's modular varieties. For Rapoport-Zink spaces of PEL type, Mantovan in \cite{Mantovan08} and Shen in \cite{Shen13} verified the conjecture assuming Hodge-Newton reducibility. For local Shimura varieties constructed by Scholze, Hansen in \cite{Hansen16} proves the conjecture for $G=\GL_n$ also under the Hodge-Newton reducibility assumption.

We now briefly sketch our proof of the theorem. Our overall strategy is to prove that the rigid analytic generic fiber of the space $\RZ_{G, b}$ is ``parabolically induced'' from the rigid analytic generic fiber of $\RZ_{\levi, b}$. More precisely, we will construct an analogue of Rapoport-Zink space $\RZ_{\parabolic, b}$ associated to the parabolic subgroup $\parabolic$ and prove the following lemma:
\begin{lemma*}
The rigid analytic generic fibers of $\RZ_{G, b}, \RZ_{\parabolic, b}$ and $\RZ_{\levi, b}$ fit into a diagram
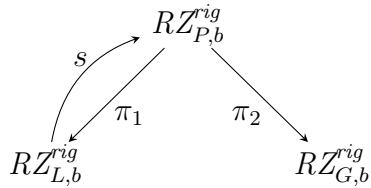
\begin{center}
\begin{tikzpicture}[description/.style={fill=white,inner sep=2pt}]
\matrix (m) [matrix of math nodes, row sep=3em,
column sep=1.5em, text height=1.5ex, text depth=0.25ex]
{  & \RZ_{\parabolic, b}^\rig & \\
\RZ_{\levi, b}^\rig& & \RZ_{G, b}^\rig\\
};
\path[-stealth]
    (m-1-2) edge node [below] {\hspace{.8em}$\pi_1$} (m-2-1)
    (m-2-1) edge[bend left] node [above] {$s$} (m-1-2)
    (m-1-2) edge node [below] {$\pi_2$\hspace{.8em}$ $} (m-2-3);
\end{tikzpicture}
\end{center}
such that
\begin{enumerate}[label=(\roman*)]
\item $s$ is a closed immersion,
\item $\pi_1$ is a fibration in balls,
\item $\pi_2$ is an isomorphism. 
\end{enumerate}
\end{lemma*}
After establishing this lemma, we deduce the theorem by comparing the cohomology of the spaces $\RZ_{\levi, b}$ and $\RZ_{G, b}$ with the cohomology of $\RZ_{\parabolic, b}$.

This strategy originated in Mantovan's proof in \cite{Mantovan08}, and also appeared in the work of Shen in \cite{Shen13} and Hansen in \cite{Hansen16}. However, details of our argument will be different from those in the aforementioned works.

The main ingredient of our argument is the notion of \emph{EL realization} developed by the author in \cite{Hong16}. An EL realization of the datum $(G, [b], \{\mu\})$ is an embedding of $(G, [b], \{\mu\})$ into a local Shimura datum of EL type $(\EL{G}, [b], \{\mu\})$ which is Hodge-Newton reducible with respect to a parabolic subgroup $\EL{\parabolic}$ of $\EL{G}$ with Levi factor $\EL{\levi}$ such that $\parabolic = \EL{\parabolic} \cap G$ and $\levi = \EL{\levi} \cap G$. By functoriality of Hodge type Rapoport-Zink spaces, an EL realization of $(G, [b], \{\mu\})$ induces a closed embedding
\[ \RZ_{G, b} \hooklongrightarrow \RZ_{\EL{G}, b}. \]
Over the EL type Rapoport-Zink space $\RZ_{\EL{G}, b}$, Mantovan constructed an analogue of Rapoport-Zink space $\RZ_{\EL{\parabolic}, b}$ associated to $\EL{\parabolic}$. We will use the results from \cite{Hong16} to prove that the pull back of $\RZ_{\EL{\parabolic}, b}$ over $\RZ_{G, b}$ is the desired space $\RZ_{\parabolic, b}$ that yields the diagram in the lemma. 

We now give an overview of the structure of this paper. In section 2, we introduce general notations and recall some group theoretic preliminaries. In section 3, we review W. Kim's construction of Rapoport-Zink spaces of Hodge type. In section 4, we state and prove our main theorem.

\subsection*{Acknowledgments} I would like to express my deepest gratitude to Elena Mantovan. This study would have never been possible without her previous work for EL/PEL cases and her numerous helpful suggestions.

\section{Notations and preliminaries}

\subsection{General notations}\label{general notations}$ $

Throughout this paper, we use the following standard notations:
\begin{itemize}
\item $\overline{\F}_p$ is a fixed algebraic closure of $\F_p$;
\item $\overline{\Q}_p$ is a fixed algebraic closure of $\Q_p$;
\item $\Q_p^\unram$ is the maximal unramified extension of $\Q_p$ in $\overline{\Q}_p$;
\item $\C_p$ is the $p$-adic completion of $\overline{\Q}_p$;
\item $\completemaxunram$ is the $p$-adic completion of $\Q_p^\unram$;
\item $\completemaxunramint$ is the ring of integers of $\completemaxunram$.
\end{itemize}
In addition, we denote by $\sigma$ the Frobenius automorphism of $\overline{\F}_p$ and also its lift to $\completemaxunramint$ and $\completemaxunram$. 



Given a Noetherian ring $\genring$ and a free $\genring$-module $\module$, we denote by $\module^\otimes$ the direct sum of all the $\genring$-modules which can be formed from $\module$ using the operations of taking duals, tensor products, symmetric powers and exterior powers. An element of $\module^\otimes$ is called a \emph{tensor} over $\module$. Note that there is a natural identification $\module^\otimes \simeq (\module^*)^\otimes$ where $\module^*$ is the dual $\genring$-module of $\module$. Any isomorphism $\module \stackrel{\sim}{\rightarrow} \module'$ of free $R$-modules of finite rank naturally induces an isomorphism $\module^\otimes \stackrel{\sim}{\rightarrow} (\module')^\otimes$.

For a $p$-divisible group $\BT$ over a $\Z_p$-scheme $\genscheme$, we write $\D(\BT)$ for its (contravariant) Dieudonn\'e module and $\Fil^1 (\D(\BT)) \subset \D(\BT)_\genscheme$ for its Hodge filtration. We generally denote by $F$ the Frobenius map on $\D(\BT)$. 

\subsection{Group theoretic preliminaries}$ $

\subsubsection{}\label{mu-filtration}
Let $G$ be a connected reductive group over $\Z_p$. We write $\text{Rep}_{\Z_p}(G)$ for the category of finite rank $G$-representations of over $\Z_p$, and $\text{Rep}_{\Q_p}(G)$ for the category of finite dimensional $G_{\Q_p}$-representations over $\Q_p$.

Let $\genring$ be a $\Z_p$-algebra, and let $\gencocharacter: \Gm \to G_\genring$ be a cocharacter. We denote by $\{\gencocharacter\}_G$, or usually by $\{\gencocharacter\}$ if there is no risk of confusion, the $G(\genring)$-conjugacy class of $\gencocharacter$. 
When $\genring=\completemaxunramint$, we have a bijection
\[ \Hom_{\completemaxunramint}(\Gm, G_{\completemaxunramint})/G(\completemaxunramint) \cong \Hom_{\completemaxunram}(\Gm, G_{\completemaxunram})/G({\completemaxunram}) \stackrel{\sim}{\longrightarrow} G(\completemaxunramint) \backslash G({\completemaxunram})/ G(\completemaxunramint)\]
induced by $\{\gencocharacter\} \mapsto G(\completemaxunramint) \gencocharacter(p) G(\completemaxunramint)$; in fact, the first bijection comes from the fact that $G$ is split over $\completemaxunramint$, whereas the second bijection is the Cartan decomposition. 


Let $\module \in \text{Rep}_{\Z_p}(G)$ be a faithful $G$-representation over $\Z_p$. By \cite{Kisin10}, Proposition 1.3.2, we can choose a finite family of tensors $(s_i)_{i \in I}$ on $\module$ such that $G$ is the pointwise stabilizer of the $s_i$; i.e., for any $\Z_p$-algebra $\genring$ we have
\[ G(\genring) = \{ g \in \GL(\module \otimes_{\Z_p} {\genring}) : g(s_i \otimes 1) = s_i \otimes 1 \text{ for all } i \in I\}.\]
We say that a grading $\gr^{\bullet} (\module_\genring)$ is \emph{induced by $\gencocharacter$} if the following conditions are satisfied:
\begin{itemize}
\item[(i)] the $\Gm$-action on $\module_\genring$ via $\gencocharacter$ leaves each grading stable,
\item[(ii)] the resulting $\Gm$-action on $\gr^i(\module_\genring)$ is given by
\[ \Gm \xrightarrow{z \mapsto z^{-i}} \Gm \xrightarrow{ z \mapsto z \cdot \text{id}} \GL(\gr^i(\module_\genring)).\]
\end{itemize}

Let $\genscheme$ be an $\genring$-scheme, and $\vectorbundle$ a vector bundle on $\genscheme$. For a finite family of global sections $(t_i)$ of $\vectorbundle^\otimes$, we define the following scheme over $\genscheme$
\[ \isommatchingtensors_{\genscheme} := \Isom_{\sheafof{\genscheme}}\Big( [\vectorbundle, (t_i)], [ \module \otimes_\genring \sheafof{\genscheme} , (s_i \otimes 1)]\Big).\]
In other words, $\isommatchingtensors_{\genscheme}$ classifies isomorphisms of vector bundles $\vectorbundle \cong \module \otimes_\genring \sheafof{\genscheme}$ which match $(t_i)$ and $(s_i \otimes 1)$.

Let $\Fil^\bullet (\vectorbundle)$ be a filtration of $\vectorbundle$. When $\isommatchingtensors_{\genscheme}$ is a trivial $G$-torsor, we say that $\Fil^\bullet (\vectorbundle)$ is a \emph{$\{ \gencocharacter\}$-filtration} with respect to $(t_i)$ if there exists an isomorphism $\vectorbundle \cong \module \otimes_\genring \sheafof{\genscheme}$, matching $(t_i)$ and $(1 \otimes s_i)$, which takes $\Fil^\bullet (\vectorbundle)$ to a filtration of $\module \otimes_\genring \sheafof{\genscheme}$ induced by $g \gencocharacter g\inv$ for some $g \in G(R)$. More generally, when $\isommatchingtensors_{\genscheme}$ a $G$-torsor, we say that $\Fil^\bullet (\vectorbundle)$  is a $\{ \gencocharacter\}$-filtration with respect to $(t_i)$ if it is \'etale-locally a $\{\gencocharacter\}$-filtration.

\subsubsection{}\label{sigma-conjugacy classes}
We say that $b, b' \in G(\completemaxunram)$ are \emph{$\sigma$-conjugate} if $b' = g b \sigma(g)\inv$ for some $g \in G(\completemaxunram)$. We denote by $B(G)$ the set of all $\sigma$-conjugacy classes in $G(\completemaxunram)$. We write $[b]_G$, or simply $[b]$ when there is no risk of confusion, for the $\sigma$-conjugacy class of $b \in G(\completemaxunram)$.

Let us now fix a $\sigma$-conjugacy class $[b]$ in $B(G)$ and choose an element $b \in [b]$. We define a group valued functor $J_b$ on the category of $\Q_p$-algebras by setting for any $\Q_p$-algebra $\genring$
\[J_b(\genring):=\{ g \in G(\genring \otimes_{\Q_p} \completemaxunram) : gb\sigma(g)\inv = b\}.\]
This functor is represented by an algebraic group over $\Q_p$ which is an inner form of some Levi subgroup of $G_{\Q_p}$ (see \cite{Rapoport-Zink96}, Corollary 1.14.). The isomorphism class of $J_b$ does not depend on the choice $b \in [b]$ since any $g \in G(\completemaxunram)$ induces an isomorphism $J_b \cong J_{gb\sigma(g)\inv}$ via conjugation.

By an \emph{$F$-isocrystal} over $\overline{\F}_p$, we mean a vector space over $\completemaxunram$ equipped with a $\sigma$-linear automorphism $F$. Given a $G_{\Q_p}$-representation $\rho: G_{\Q_p} \longrightarrow \GL(V)$ over $\Q_p$, we set $N_b(\rho)$ to be the $\completemaxunram$-vector space $V \otimes_{\Q_p} \completemaxunram$ with $F=\rho(b) \circ (1 \otimes \sigma)$. Then $N_b$ defines a functor from $\text{Rep}_{\Q_p}(G)$ to the category of $F$-isocrystals over $\overline{\F}_p$. One easily checks that another choice $b' \in [b]$ gives an isomorphic functor $N_{b'}$.


\section{Rapoport-Zink spaces of Hodge type}\label{review of rapoport-zink spaces of hodge type}
In this section, we discuss the construction and key properties of Rapoport-Zink spaces of Hodge type, following \cite{Kim13}.

\subsection{Construction}
\subsubsection{} \label{definition of unram local Shimura data of Hodge type}

An \emph{unramified local Shimura datum of Hodge type} is a tuple $(G, [b], \{\mu\})$ where
\begin{itemize}
\item $G$ is a connected reductive group over $\Z_p$;
\item $[b]$ is a $\sigma$-conjugacy class of $G(\completemaxunram)$;
\item $\{\mu\}$ is a $G(\completemaxunramint)$-conjugacy class of cocharacters of $G$,
\end{itemize}
which satisfy the following conditions:
\begin{enumerate}[label=(\roman*)]
\item\label{unram loc shimura datum condition 1} $\{\mu\}$ is minuscule, 
\item\label{unram loc shimura datum condition 2} $[b] \cap G(\completemaxunramint) \mu(p) G(\completemaxunramint)$ is not empty for some (and hence for all) $\mu \in \{\mu\}$, 
\item\label{unram loc shimura datum condition 3} there exists a faithful $G$-representation $\module \in \text{Rep}_{\Z_p}(G)$ (with its dual $\module^*$) such that, for some $b \in [b] \cap G(\completemaxunramint) \mu(p) G(\completemaxunramint)$, we have a $\completemaxunramint$-lattice 
\[\dieudonnemodule \simeq \module^* \otimes_{\Z_p} \completemaxunramint \subset  N_b(\module^* \otimes_{\Z_p} \Q_p)\] 
with the property $p \dieudonnemodule \subset F \dieudonnemodule \subset \dieudonnemodule$. 
\end{enumerate}
Note that the set $G(\completemaxunramint) \mu(p) G(\completemaxunramint)$ is independent of the choice $\mu \in \{\mu\}$ as explained in \ref{mu-filtration}

Condition \ref{unram loc shimura datum condition 3} implies that for \emph{all} $b \in [b] \cap G(\completemaxunramint) \mu(p) G(\completemaxunramint)$, we can find a $\completemaxunramint$-lattice \[\dieudonnemodule \simeq \module^* \otimes_{\Z_p} \completemaxunramint \subset  N_b(\module^* \otimes_{\Z_p} \Q_p)\] 
with the property $p \dieudonnemodule \subset F \dieudonnemodule \subset \dieudonnemodule$. In fact, existence of $\dieudonnemodule$ is equivalent to the condition that the linearization of $F$ has an integer matrix representation with respect to some basis, which depends only on $[b]$.

Let us explain how the above definition is related to the definition of local Shimura data given by Rapoport and Viehmann in \cite{Rapoport-Viehmann14}, Definition 5.1. Since $G$ is split over $\completemaxunramint$, we may regard $\{\mu\}$ as a geometric conjugacy class of cocharacters of $G$. Then by the work of Kottwitz-Rapoport \cite{Kottwitz-Rapoport03}, Lucarelli \cite{Lucarelli04} and Gashi \cite{Gashi10}, we can state the condition \ref{unram loc shimura datum condition 2} as $[b] \in B(G_{\Q_p}, \{\mu\})$ where $B(G_{\Q_p}, \{\mu\})$ is the Kottwitz set defined in \cite{Kottwitz97}. Hence the tuple $(G_{\Q_p}, [b], \{\mu\})$ is a local Shimura datum as defined in \cite{Rapoport-Viehmann14}, Definition 5.1. 

\begin{lemma} \label{functoriality of unram local shimura data of hodge type} 
Let $(G, [b], \{\mu\})$ be an unramified local Shimura datum of Hodge type. 
\begin{enumerate} 
\item\label{product functoriality of unram local shimura data of hodge type} For any unramified local Shimura datum of Hodge type $(G', [b'], \{\mu'\})$, the tuple $(G \times G', [b, b'], \{\mu, \mu'\})$ is also an unramified local Shimura datum of Hodge type.
\item\label{hom functoriality of unram local shimura data of hodge type} For any homomorphism $f: G \longrightarrow G'$ of connected reductive group over $\Z_p$, the tuple $(G', [f(b)], \{f \circ \mu\})$ is an unramified local Shimura datum of Hodge type. 
\end{enumerate}
\end{lemma}
\begin{proof}
This is an easy consequence of definition. 
\end{proof}

\subsubsection{} \label{description of F-crystals with tensors}

For the rest of this section, we fix our unramified local Shimura datum of Hodge type $(G, [b], \{\mu\})$ and also a faithful $G$-representation $\module \in \text{Rep}_{\Z_p}(G)$ in the condition \ref{unram loc shimura datum condition 3} of \ref{definition of unram local Shimura data of Hodge type}. By Lemma \ref{functoriality of unram local shimura data of hodge type}, we obtain a morphism of unramified local Shimura data of Hodge type
\[(G, [b], \{\mu\}) \longrightarrow (\GL(\module), [b]_{\GL(\module)}, \{\mu\}_{\GL(\module)}).\]

Let us now choose an element $b \in [b] \cap G(\completemaxunramint) \mu(p) G(\completemaxunramint)$ and take $\dieudonnemodule \simeq \module^* \otimes_{\Z_p} \completemaxunramint$ as in the condition \ref{unram loc shimura datum condition 3} of \ref{definition of unram local Shimura data of Hodge type}. We also choose a finite family of tensors $(s_i)_{i \in I}$ on $\module$ as in \ref{mu-filtration}. Then $\dieudonnemodule \simeq \module^* \otimes_{\Z_p} \completemaxunramint$ is equipped with tensors $(t_i) := (s_i \otimes 1)$, which are $F$-invariant since the linearization of $F$ on $\dieudonnemodule[1/p] = N_b(\module^* \otimes_{\Z_p} \Q_p)$ is given by $b \in G(\completemaxunram)$.

Take $\BT$ to be a $p$-divisible group over $\overline{\F}_p$ with $\D(\BT) = \dieudonnemodule$. Then the Hodge filtration $\Fil^1 (\D(\BT)) \subset \D(\BT)$ is a $\{\sigma\inv(\mu\inv)\}$-filtration with respect to $(t_i)$ (see \cite{Kim13}, Lemma 2.5.7 and Remark 2.5.8.). We may regard the tensors $(t_i)$ as additional structures on $\BT$ induced by the group $G$. We will often refer to these additional structures as \emph{$G$-structure} on $\BT$. We will write $\generalhodge{\BT} :=(\BT, (t_i))$ to indicate the $p$-divisible group $\BT$ with $G$-structure.

\subsubsection{} \label{rapoport-zink space for GLn}

Let $\Nilp_{\completemaxunramint}$ denote the category of $\completemaxunramint$-algebra where $p$ is nilpotent. For any $\genring \in \Nilp_{\completemaxunramint}$ we set $\RZ_b(\genring)$ to be the set of isomorphism classes of pairs $(\lift{\BT}, \genqisog)$ where
\begin{itemize}
\item $\lift{\BT}$ is a $p$-divisible group over $\genring$;
\item $\genqisog: \BT_{\genring/p} \longrightarrow \lift{\BT}_{\genring/p}$ is a quasi-isogeny, i.e., an invertible global section of \newline$\Hom(\BT_{\genring/p}, \lift{\BT}_{\genring/p}) \otimes_\Z \Q$. 
\end{itemize}
Then $\RZ_b$ defines a covariant set-valued functor on $\Nilp_{\completemaxunramint}$, which does not depend on the choice of $b \in [b] \cap G(\completemaxunramint) \mu(p) G(\completemaxunramint)$ up to isomorphism. 
Rapoport and Zink in \cite{Rapoport-Zink96} proved that the functor $\RZ_b$ is represented by a formal scheme which is locally formally of finite type and formally smooth over $\completemaxunramint$. We write $\RZ_b$ also for the representing formal scheme, and $\lift{\BT}_{\GL, b}$ for the universal $p$-divisible group over $\RZ_b$.

\subsubsection{}\label{rapoport-zink space of hodge type as a functor}

Given a pair $(\lift{\BT}, \genqisog) \in \RZ_b(\genring)$ with $\genring \in \Nilp_{\completemaxunramint}$, we have an isomorphism 
\[ \D(\genqisog) : \D(\lift{\BT}_{\genring/p}) [1/p] \stackrel{\sim}{\longrightarrow}  \D(\BT_{\genring/p})[1/p] \]
induced by $\genqisog$. We write $(t_{\lift{\BT}, i})$ for the inverse image of the tensors $(t_i)_\genring$ under this isomorphism.

Let $\Nilp_{\completemaxunramint}^\smooth$ denote the full subcategory of $\Nilp_{\completemaxunramint}$ consisting of formally smooth and formally finitely generated algebra over $\completemaxunramint/p^m$ for some positive integer $m$. For any $\genring \in \Nilp_{\completemaxunramint}^\smooth$, we define the set $\RZ_{G, b}^{(s_i)}(\genring) \subset \Hom_{\completemaxunramint}(\Spf (\genring), \RZ_b)$ as follows: for a morphism $f: \Spf(\genring) \to \RZ_b$ and a $p$-divisible group $\lift{\BT}$ over $\Spec(\genring)$ which pulls back to $f^* \lift{\BT}_{\GL, b}$ over $\Spf (\genring)$, we have $f \in \RZ_{G, b}^{(s_i)}(\genring)$ if and only if there exists a (unique) family of tensors $(\liftlow{t}_i)$ on $\D(\lift{\BT})$ with the following properties:
\begin{enumerate}[label=(\roman*)]
\item\label{Hodge type RZ functor isogeny condition} for some ideal of definition $J$ of $\genring$ containing $p$, the pull-back of $(\liftlow{t}_i)$ over $\genring/J$ agrees with the pull-back of $(t_{\lift{\BT}, i})$ over $\genring/J$,
\item\label{Hodge type RZ functor G-action condition} for a $p$-adic lift $\lift{\genring}$ of $\genring$ which is formally smooth over $\completemaxunramint$, the $\lift{\genring}$-scheme
\[ \isommatchingtensors_{\lift{\genring}} := \Isom_{\lift{\genring}}\Big( [\D(\lift{\BT})_{\lift{\genring}}, (\liftlow{t}_i)_{\lift{\genring}}], [ \module^* \otimes_{\Z_p} \lift{\genring} , (s_i \otimes 1)]\Big)\]
defined as in \ref{mu-filtration} is a $G$-torsor,
\item\label{Hodge type RZ functor hodge filtration condition} the Hodge filtration of $\lift{\BT}$ is a $\{\sigma\inv(\mu\inv)\}$-filtration with respect to $(\liftlow{t}_i)$. 
\end{enumerate}
Then $\RZ_{G, b}^{(s_i)}$ defines a set-valued functor on $\Nilp_{\completemaxunramint}^\smooth$.


\subsubsection{}\label{closed points of RZ functor}
Let us give a concrete description of the set $\RZ_{G, b}^{(s_i)}(\overline{\F}_p)$. Consider a pair $(\lift{\BT}, \genqisog) \in \RZ_b(\overline{\F}_p)$ with a family of tensors $(\liftlow{t}_i)$ on $\D(\lift{\BT})$. Then $(\liftlow{t}_i)$ has the property \ref{Hodge type RZ functor isogeny condition} of \ref{rapoport-zink space of hodge type as a functor} if and only if it is matched with the family $(t_i)$ under the isomorphism 
\[ \D(\genqisog) : \D(\lift{\BT}) [1/p] \stackrel{\sim}{\longrightarrow}  \D(\BT)[1/p] \]
induced by $\genqisog$. In addition, it satisfies the properties \ref{Hodge type RZ functor G-action condition} and \ref{Hodge type RZ functor hodge filtration condition} of \ref{rapoport-zink space of hodge type as a functor} if and only if $(\lift{\BT}, (\liftlow{t}_i))$ is a $p$-divisible group with $G$-structure that arises from the datum $(G, [b], \{\mu\})$. Hence the set $\RZ_{G, b}^{(s_i)}(\overline{\F}_p)$ classifies the isomorphism classes of tuples $(\lift{\BT}, (\liftlow{t}_i), \genqisog)$ where
\begin{itemize}
\item $(\lift{\BT}, (\liftlow{t}_i))$ is a $p$-divisible group over $\overline{\F}_p$ with $G$-structure;
\item $\genqisog: \BT \longrightarrow \lift{\BT}$ is a quasi-isogeny such that the induced isomorphism $\D(\lift{\BT})[1/p] \stackrel{\sim}{\longrightarrow} \D(\BT)[1/p]$ matches $(\liftlow{t}_i)$ with $(t_i)$.
\end{itemize}

\begin{prop}[\cite{Kim13}, Theorem 4.9.1.] \label{rapoport-zink spaces of hodge type construction} Assume that $p>2$. Then there exists a closed formal subscheme $\RZ_{G, b} \subset \RZ_b$, which is formally smooth over $\completemaxunramint$ and represents the functor $\RZ_{G, b}^{(s_i)}$ for any choice of the tensors $(s_i)$ in \ref{description of F-crystals with tensors}. Moreover, the isomorphism class of the formal scheme $\RZ_{G, b}$ depends only on the datum $(G, [b], \{\mu\})$. 
\end{prop}
We let $\lift{\BT}_{G, b}$ denote the ``universal $p$-divisible group"  over $\RZ_{G, b}$, obtained by taking the pull-back of $\lift{\BT}_{\GL, b}$. Then we obtain a family of ``universal tensors" $(\liftlow{t}^\univ_i)$ on $\D(\lift{\BT}_{G, b})$ by applying the universal property to an open affine covering of $\RZ_{G, b}$.

\begin{example}\label{rapoport-zink spaces of EL type}
Consider the case $G = \ELgroupintegral$ where $\integerring$ is the ring of integers of some finite unramified extension of $\Q_p$. In this case, choosing a family of tensors $(s_i)$ on $\module$ as in \ref{description of F-crystals with tensors} is equivalent to choosing a $\Z_p$-basis of $\integerring$. Then the family of tensors $(t_i)$ encodes an action of $\integerring$ on $\dieudonnemodule$ and thus on $\BT$. Hence $\generalhodge{\BT} =(\BT, (t_i))$ can be identified with a $p$-divisible group $\BT$ with an action of $\integerring$. 

In this setting, the construction of $\RZ_{G, b}$ agrees with the construction of Rapoport-Zink spaces of EL type in \cite{Rapoport-Zink96} (see \cite{Kim13}, Proposition 4.7.1.). In other words, for any $\genring \in \Nilp_{\completemaxunramint}$ the set $\RZ_{G, b}(\genring)$ classifies the isomorphism classes of pairs $(\lift{\BT}, \genqisog)$ where
\begin{itemize}
\item $\lift{\BT}$ is a $p$-divisible group over $\genring$, endowed with an action of $\integerring$ such that
\[ \det {}_\genring(a, \Lie(\lift{\BT})) = \det(a, \Fil^0(\D(\BT))_{\completemaxunram}) \quad \text{ for all } a \in \integerring,\]
\item $\genqisog: \BT_{\genring/p} \to \lift{\BT}_{\genring/p}$ is a quasi-isogeny which commutes with the action of $\integerring$.
\end{itemize}

\end{example}

\subsection{Functorial properties}$ $

For the rest of this section, we assume that $p>2$ and take $\RZ_{G, b}$ as in Proposition \ref{rapoport-zink spaces of hodge type construction}.


\begin{prop}[\cite{Kim13}, Theorem 4.9.1.] \label{functoriality of rapoport-zink spaces}
Let $(G', [b'], \{\mu'\})$ be another unramified local Shimura datum of Hodge type, and choose $b' \in [b'] \cap G(\completemaxunramint) \mu'(p) G(\completemaxunramint)$ that gives rise to a $p$-divisible group over $\overline{\F}_p$ with $G'$-structure as in \ref{description of F-crystals with tensors}. 

\begin{enumerate}[label=(\arabic*)]
\item\label{product functoriality of RZ spaces of hodge type} The natural morphism $\RZ_b \times_{\Spf (\completemaxunramint)} \RZ_{b'} \longrightarrow \RZ_{(b, b')}$, defined by the product of $p$-divisible groups with quasi-isogeny, induces an isomorphism
\begin{equation*} \label{product natural morphism on RZ spaces}
\RZ_{G, b} \times_{\Spf (\completemaxunramint)} \RZ_{G', b'} \stackrel{\sim}{\longrightarrow} \RZ_{G \times G', (b, b')}
\end{equation*}

\item\label{hom functoriality of RZ spaces of hodge type} For any homomorphism $f : G \longrightarrow G'$ with $f(b) = b'$, there exists an induced morphism 
\begin{equation*} \label{hom natural morphism on RZ spaces}
\RZ_{G, b} \longrightarrow \RZ_{G', b'},
\end{equation*}
which is a closed embedding if $f$ is a closed embedding. 
\end{enumerate}
\end{prop}

\subsubsection{} \label{closed points of RZ spaces and affine Deligne Lusztig set}

We want to describe the functorial properties in Proposition \ref{functoriality of rapoport-zink spaces} on the set of $\overline{\F}_p$-valued points. For this, we introduce the set
\[ \DeligneLusztig{G}{ \mu}{b} := \{ g \in G(\completemaxunram)/ G(\completemaxunramint) | g b \sigma(g)\inv \in G(\completemaxunramint) \mu(p) G(\completemaxunramint)\} \]
which is clearly independent of our choice of $b \in [b]$ up to bijection. The set $\DeligneLusztig{G}{ \mu}{b}$ is called the \emph{affine Deligne-Lusztig set} associated to the datum $(G, [b], \{\mu\})$. As explained in \cite{Kim13}, 4.8, we have a natural bijection
\begin{equation*}\label{k-point of rapoport-zink spaces} \DeligneLusztig{G}{ \mu}{ b} \stackrel{\sim}{\longrightarrow} \RZ_{G, b}(\overline{\F}_p).\end{equation*}

Let us now consider another unramified local Shimura datum $(G', [b'], \{\mu'\})$ and choose $b' \in [b'] \cap G(\completemaxunramint) \mu'(p) G(\completemaxunramint)$ as in Proposition \ref{functoriality of rapoport-zink spaces}. Then on the set of $\overline{\F}_p$-valued points, the morphism in \ref{product functoriality of RZ spaces of hodge type} of Proposition \ref{functoriality of rapoport-zink spaces} gives a map
\[\DeligneLusztig{G}{ \mu}{ b} \times \DeligneLusztig{G'}{ \mu'}{ b'} \stackrel{\sim}{\longrightarrow}\DeligneLusztig{G \times G'}{ \mu , \mu'}{ b, b'} \]
which maps $\big(g G(\completemaxunramint), g' G'(\completemaxunramint)\big)$ to $(g, g') (G \times G')(\completemaxunramint)$. 
For any homomorphism $f : G \longrightarrow G'$ with $f(b) = b'$, the morphism in \ref{hom functoriality of RZ spaces of hodge type} of Proposition \ref{functoriality of rapoport-zink spaces} yields a map
\[\DeligneLusztig{G}{ \mu}{ b} \longrightarrow \DeligneLusztig{G'}{ f \circ \mu}{ f(b)}\]
which maps $gG(\completemaxunramint)$ to $f(g)G'(\completemaxunramint)$.

\subsubsection{}

We now describe the functorial properties in Proposition \ref{functoriality of rapoport-zink spaces} on the formal completions at an  $\overline{\F}_p$-valued point. Let $x$ be a point in $\RZ_{G, b}(\overline{\F}_p)$, and write $(\BT_x, (t_{x, i}), \genqisog_x)$ for the corresponding tuple under the description of $\RZ_{G, b}(\overline{\F}_p)$ in \ref{closed points of RZ functor}. We denote by $\completion{(\RZ_{G, b})_x}$ the formal completion of $\RZ_{G, b}$ at $x$.

For an artinian local $\completemaxunramint$-algebra $\genring$ with residue field $\overline{\F}_p$, we define a \emph{deformation} of $\BT_x$ over $\genring$ to be a $p$-divisible group $\deform{\BT}_x$ over $\genring$ with an isomorphism $\deform{\BT}_x \otimes_\genring \overline{\F}_p \cong \BT_x$. By Faltings in \cite{Faltings99}, \S7, there exists a formal scheme $\Def_{\BT_x, G}$ over $\Spf (\completemaxunramint)$ which classifies the deformations of $\BT_x$ with Tate tensors in the following sense: for a formally smooth $\completemaxunramint$-algebra of the form $\genring = \completemaxunramint[[u_1, \cdots, u_N]]$ or $\genring = \completemaxunramint[[u_1, \cdots, u_N]]/(p^m)$, $\Def_{\BT_x, G}(\genring)$ is the set of isomorphism classes of the pair $(\deform{\BT}_x, (\liftlow{t}_i))$ where
\begin{itemize}
\item $\deform{\BT}_x$ is a deformation of $\BT_x$ over $\genring$;
\item $(\liftlow{t}_i)$ is a family of Frobenius-invariant tensors on $\D(\deform{\BT}_x)$ which lift the tensors $(t_i)$ and lie in the $0$th filtration with respect to the Hodge filtration. 
\end{itemize}
From this moduli description, we obtain a natural isomorphism
\[ \Def_{\BT_x, G} \simeq \completion{(\RZ_{G, b})_x}\]
as explained in \cite{Kim13}, 4.8.

Now consider another unramified local Shimura datum $(G', [b'], \{\mu'\})$ and choose $b' \in [b'] \cap G(\completemaxunramint) \mu'(p) G(\completemaxunramint)$ as in Proposition \ref{functoriality of rapoport-zink spaces}. For any point $x' \in \RZ_{G', b'}(\overline{\F}_p)$, the morphism in \ref{product functoriality of RZ spaces of hodge type} of Proposition \ref{functoriality of rapoport-zink spaces} induces an isomorphism
\[\Def_{\BT_x, G} \times \Def_{\BT_{x'}, G'} \stackrel{\sim}{\longrightarrow} \Def_{\BT_x \times \BT_{x'}, G \times G'}  \]
defined by the product of deformations. For any homomorphism $f : G \longrightarrow G'$ with $f(b) = b'$, if we take $x' \in \RZ_{G', b'}(\overline{\F}_p)$ to be the image of $x$ under the morphism in \ref{hom functoriality of RZ spaces of hodge type} of Proposition \ref{functoriality of rapoport-zink spaces}, we have an induced morphism
\[ \Def_{\BT_x, G} \longrightarrow \Def_{\BT_{x'}, G'}\]
which is a closed embedding if $f$ is a closed embedding.

\subsection{Associated local Shimura varieties}

\subsubsection{} \label{action of Jb on rapoport-zink spaces}
Consider the algebraic group $J_b$ over $\Q_p$ defined in \ref{sigma-conjugacy classes}. Note that $J_b(\Q_p)$ can be identified with the group of quasi-isogenies $\gamma: \BT \longrightarrow \BT$ that preserve the tensors $(t_i)$. One can show that $\RZ_{G, b}$ carries a natural left $J_b(\Q_p)$-action defined by
\[ \gamma (\lift{\BT}, \genqisog) = (\lift{\BT}, \genqisog \circ \gamma\inv)\]
for any $\genring \in \Nilp_{\completemaxunramint}, (\lift{\BT}, \genqisog) \in \RZ_{G, b}(\genring)$ and $\gamma \in J_b(\Q_p)$ (see \cite{Kim13}, 7.2.).

\subsubsection{} \label{weil descent datum on rapoport-zink spaces}

Let $\localreflexfield$ be the field of definition of the $G(\completemaxunram)$-conjugacy class of $\mu$, and let $\integerring_{\localreflexfield}$ denote its ring of integers. Note that $\localreflexfield$ is a finite unramified extension of $\Q_p$ since $G_{\Q_p}$ is split over a finite unramified extension of $\Q_p$. Let $d$ be the degree of the extension, and write $\tau$ for the Frobenius automorphism of $\completemaxunram$ relative to $\localreflexfield$.

For any formal scheme $\genscheme$ over $\Spf (\completemaxunramint)$, we write $\genscheme^\tau : = \genscheme \times_{\Spf (\completemaxunramint), \tau} \Spf (\completemaxunramint)$. 
By a \emph{Weil descent datum} on $\genscheme$ over $\integerring_\localreflexfield$, we mean an isomorphism $\genscheme \stackrel{\sim}{\longrightarrow} \genscheme^\tau$. If $\genscheme \cong \genscheme_0 \times_{\Spf(\integerring_\localreflexfield)} \Spf(\completemaxunramint)$ for some  formal scheme $\genscheme_0$ over $\Spf (\integerring_\localreflexfield)$, then there exists a natural Weil descent datum on $\genscheme$ over $\integerring_\localreflexfield$, called an \emph{effective} Weil descent datum.

For any $\genring \in \Nilp_{\completemaxunramint}$, we define $\genring^\tau$ to be $\genring$ viewed as a $\completemaxunramint$-algebra via $\tau$. Note that we have a natural identification $\RZ_b^\tau(\genring) = \RZ_b(\genring^\tau)$. Following Rapoport and Zink in \cite{Rapoport-Zink96}, 3.48, we define a Weil descent datum $\weildescent$ on $\RZ_b$ over $\integerring_\localreflexfield$ by sending $(\lift{\BT}, \genqisog) \in \RZ_b(\genring)$ with $\genring \in \Nilp_{\completemaxunramint}$ to $(\lift{\BT}^\weildescent, \genqisog^\weildescent) \in \RZ_b(\genring^\tau)$ where
\begin{itemize}
\item $\lift{\BT}^\weildescent$ is $\lift{\BT}$ viewed as a $p$-divisible group over $\genring^\tau$;
\item $\genqisog^\weildescent$ is the quasi-isogeny
\[ \genqisog^\weildescent : \BT_{\genring^\tau /p} = (\tau^* \BT)_{\genring/p} \xrightarrow{\text{Frob}^{-d}} \BT_{R/p} \stackrel{\genqisog}{\longrightarrow} \lift{\BT}_{\genring/p} = \lift{\BT}^\weildescent_{\genring/p}\]
where $\text{Frob}^d: \BT \to \tau^* \BT$ is the relative $q$-Frobenius with $q=p^d$. 
\end{itemize}
One can check that $\weildescent$ restricts to a Weil descent datum $\weildescent_G$ on $\RZ_{G, b}$ over $\integerring_\localreflexfield$ by looking at $\overline{\F}_p$-points and the formal completions thereof. The Weil descent datum $\weildescent_G$ clearly commutes with the $J_b(\Q_p)$-action defined in \ref{action of Jb on rapoport-zink spaces}. 


\subsubsection{} \label{construction of rigid analytic tower}

Since $\RZ_{G, b}$ is locally formally of finite type over $\Spf (\completemaxunramint)$, it admits a rigid analytic generic fiber which we denote by $\RZ_{G, b}^\rig$ (see \cite{Berthelot96}.). The $J_b(\Q_p)$-action and the Weil descent datum $\weildescent_G$ on $\RZ_{G, b}$ induce an action of $J_b(\Q_p)$ on $\RZ_{G, b}^\rig$ and an Weil descent datum $\varPhi_G : \RZ_{G, b}^\rig \stackrel{\sim}{\longrightarrow} (\RZ_{G, b}^\rig)^\tau$ over $\localreflexfield$. 


Recall that we have a universal $p$-divisible group $\lift{\BT}_{G, b}$ over $\RZ_{G, b}$ and a family of universal tensors $(\liftlow{t}^\univ_i)$ on $\D(\lift{\BT}_{G, b})$. In addition, the family $(\liftlow{t}^\univ_i)$ has a ``\'etale realization'' $(\liftlow{t}^\univ_{i, \et})$ on the Tate module $\tatemodule{\lift{\BT}_{G, b}}$ (see \cite{Kim13}, Theorem 7.1.6.). 

For any open compact subgroup $\levelatp$ of $G(\Z_p)$, we define the following rigid analytic \'etale cover of $\RZ_{G, b}^\rig$:
\[\RZ_{G, b}^\levelatp := \Isom_{\RZ_{G, b}^\rig}\Big( [\module, (s_i)], [ \tatemodule{\lift{\BT}_{G, b}} , (\liftlow{t}^\univ_{i, \et}) ] \Big)\Big/\levelatp.\]
The $J_b(\Q_p)$-action and the Weil descent datum over $\localreflexfield$ on $\RZ_{G, b}^\rig$ pull back to $\RZ_{G, b}^\levelatp$. 
As the level $\levelatp$ varies, these covers form a tower $\{\RZ_{G, b}^\levelatp\}$ with Galois group $G(\Z_p)$. We denote this tower by $\analtower{\RZ_{G, b}}$.

By \cite{Kim13}, Proposition 7.4.8, there exists a right $G(\Q_p)$-action on the tower $\analtower{\RZ_{G, b}}$ extending the Galois action of $G(\Z_p)$, which commutes with the natural $J_b(\Q_p)$-action and the Weil descent datum over $\localreflexfield$. In addition, there is a well-defined period map on $\RZ_{G, b}^\rig$ as explained in \cite{Kim13}, 7.5. Hence the tower $\analtower{\RZ_{G, b}}$ is a local Shimura variety in the sense of Rapoport and Viehmann in \cite{Rapoport-Viehmann14}, 5.1.

\subsubsection{} \label{cohomology of rigid analytic tower}

We fix a prime $l \neq p$, and let $\weilgroup{\localreflexfield}$ denote the Weil group of $\localreflexfield$. For any level $\levelatp \subset G(\Z_p)$, we consider the cohomology groups
\[ H^i(\RZ_{G, b}^\levelatp) = H^i_c (\RZ_{G, b}^\levelatp \otimes_{\completemaxunram} \C_p, \Q_l( \dim \RZ_{G, b}^\levelatp)).\]
As the level $\levelatp$ varies, these cohomology groups form a tower $\{ H^i(\RZ_{G, b}^\levelatp)\}$ for each $i$, endowed with a natural action of $G(\Q_p) \times \weilgroup{\localreflexfield} \times J_b(\Q_p)$.

Let $\rho$ be an admissible $l$-adic representation of $J_b(\Q_p)$. The groups 
\[H^{i, j}(\analtower{\RZ_{G, b}})_\rho := \varinjlim_{\levelatp} \Ext^j_{J_b(\Q_p)} (H^i(\RZ_{G, b}^\levelatp), \rho)\]
satisfy the following properties (see \cite{Rapoport-Viehmann14}, Proposition 6.1 and \cite{Mantovan08}, Theorem 8):
\begin{enumerate}
\item The groups $H^{i, j}(\analtower{\RZ_{G, b}})_\rho$ vanish for almost all $i, j$. 
\item There is a natural action of $G(\Q_p) \times \weilgroup{\localreflexfield}$ on each $H^{i, j}(\analtower{\RZ_{G, b}})_\rho$. 
\item The representations $H^{i, j}(\analtower{\RZ_{G, b}})_\rho$ are admissible. 
\end{enumerate}
Hence we can define a virtual representation of $G(\Q_p) \times \weilgroup{\localreflexfield}$
\[ H^\bullet(\RZ^\infty_{G, b})_\rho := \sum_{i, j \geq 0} (-1)^{i+j} H^{i, j} (\analtower{\RZ_{G, b}})_\rho.\]

\section{Hodge-Newton reducibility and Harris-Viehmann conjecture}\label{Hodge-Newton reducibility and Harris-Viehmann conjecture}


\subsection{Harris-Viehmann conjecture: statement}

\subsubsection{}\label{notations for HV conjecture}
Throughout this section, we fix a prime $p>2$ and an unramified local Shimura datum of Hodge type $(G, [b], \{\mu\})$. We also choose a faithful $G$-representation $\module \in \text{Rep}_{\Z_p}(G)$ and a finite family of tensors $(s_i)$ on $\module$ as in \ref{description of F-crystals with tensors}. In addition, we fix a maximal torus $T \subseteq G$ and a Borel subgroup $B \subseteq G$ containing $T$, both defined over $\Z_p$.

Let $\parabolic$ be a proper standard parabolic subgroup of $G$ with Levi factor $\levi$ and unipotent radical $\unipotent$. For any element $b \in [b] \cap \levi(\completemaxunram)$, we define $I_{b, \{\mu\}, \levi}$ to be the set of $\levi(\completemaxunramint)$-conjugacy classes of cocharacters of $\levi$ with a representative $\mu'$ such that
\begin{enumerate}[label=(\roman*)]
\item $\mu' \in \{\mu\}_G$,
\item $[b]_\levi \cap \levi(\completemaxunramint)\mu'(p)\levi(\completemaxunramint)$ is not empty. 
\end{enumerate}
Then $I_{b, \{\mu\}, \levi}$ is finite and nonempty (see \cite{Rapoport-Viehmann14}, Lemma 8.1.).

\begin{lemma}
For any $\{\mu'\}_\levi \in I_{b, \{\mu\}, \levi}$, the tuple $(\levi, [b]_\levi, \{\mu'\}_\levi)$ is an unramified local Shimura datum of Hodge type. 
\end{lemma}
\begin{proof}
By construction, the tuple $(\levi, [b]_\levi, \{\mu'\}_\levi)$ satisfies the conditions \ref{unram loc shimura datum condition 1} and \ref{unram loc shimura datum condition 2} of \ref{definition of unram local Shimura data of Hodge type}. Hence it remains to check 
the condition \ref{unram loc shimura datum condition 3} of \ref{definition of unram local Shimura data of Hodge type}. After taking $\sigma$-conjugate in $\levi(\completemaxunram)$ if necessary, we may assume that $b \in \levi(\completemaxunramint)\mu'(p)\levi(\completemaxunramint)$. Then we have $b \in G(\completemaxunramint)\mu(p)G(\completemaxunramint)$ since $\mu' \in \{\mu\}$. Now we verify the condition \ref{unram loc shimura datum condition 3} with $b$ since $(G, [b], \{\mu\})$ is an unramified local Shimura datum of Hodge type. 
\end{proof}

We can now state the Harris-Viehmann conjecture in the setting of Rapoport-Zink spaces of Hodge type.

\begin{conj}[\cite{Rapoport-Viehmann14}, Conjecture 8.4.] \label{HV conjecture general form}
Choose an element $b \in [b] \cap G(\completemaxunramint)\mu(p)G(\completemaxunramint)$. Let $\parabolic$ be a parabolic subgroup of $G$ with Levi factor $\levi$ such that 
\begin{enumerate}[label=(\roman*)]
\item\label{HV conjecture nonemptiness condition} $[b] \cap \levi(\completemaxunram)$ is not empty, 
\item\label{HV conjecture Jb condition} $J_b$ is an inner form of a Levi subgroup of $G$ contained in $\levi$.  
\end{enumerate}
Choose representatives $\mu_1, \mu_2, \cdots, \mu_s$ of the $\levi(\completemaxunramint)$-conjugacy classes of cocharacters in $I_{b, \{\mu\}, \levi}$, and also choose $b_k \in [b]_\levi \cap \levi(\completemaxunramint)\mu_k(p)\levi(\completemaxunramint)$ for each $k=1, 2, \cdots, s$. Then for any admissible $\overline{\Q}_l$-representation $\rho$ of $J(\Q_p)$, we have an equality of virtual representations of $G(\Q_p) \times \weilgroup{\localreflexfield}$
\[ H^\bullet(\RZ^\infty_{G, b})_\rho = \bigoplus_{k=1}^s \text{Ind}_{\parabolic(\Q_p)}^{G(\Q_p)} H^\bullet(\RZ^\infty_{\levi, b_k})_\rho.\]
In particular, the virtual representation $H^\bullet(\RZ^\infty_{G, b})_\rho$ contains no supercuspidal representations of $G(\Q_p)$.
\end{conj}

Here we consider the groups $H^\bullet(\RZ^\infty_{\levi, b_k})_\rho$ as a virtual representation of $\parabolic(\Q_p) \times \weilgroup{\localreflexfield}$ by letting the unipotent radical of $\parabolic(\Q_p)$ act trivially. Note that the choice of $b_k$'s (or $\mu_k$'s) is unimportant since the isomorphism class of the spaces $\RZ^\infty_{\levi, b_k}$ only depend on the tuples $(\levi, [b]_\levi, \{\mu_k\}_\levi)$.

\subsubsection{} \label{definition of hodge-newton reducibility}

We will prove Conjecture \ref{HV conjecture general form} under the assumption that the datum $(G, [b], \{\mu\})$ is \emph{Hodge-Newton reducible} (with respect to $\parabolic$ and $\levi$). By definition, this means that there exist $\mu \in \{\mu\}$ and $b \in [b] \cap \levi(\completemaxunram)$ with the following properties:
\begin{enumerate}[label=(\roman*)]
\item\label{hodge-newton type mu factor condition} the cocharacter $\mu$ factors through $\levi$, 
\item\label{hodge-newton type levi condition} $[b]_\levi \cap \levi(\completemaxunramint)\mu(p)\levi(\completemaxunramint)$ is not empty,
\item\label{hodge-newton type cocharacters condition} in the action of $\mu$ and $\newtonmap_b$ on $\Lie (\unipotent) \otimes_{\Q_p} \completemaxunram$, only non-negative characters occur.
\end{enumerate}
Here $\newtonmap_b$ denote the \emph{Newton cocharacter} associated to $b$ (see \cite{Kottwitz85}, \S4 or \cite{Rapoport-Richartz96}, \S1 for definition.). Note that the properties \ref{hodge-newton type mu factor condition} and \ref{hodge-newton type levi condition} together imply that $\{\mu\}_\levi \in I_{b, \{\mu\}, \levi}$.

The notion of Hodge-Newton reducibility first appeared in \cite{Katz79}, where Katz considered $p$-divisible groups (and $F$-crystals) with the property that the Hodge polygon passes through a break point of the Newton polygon. For $G= \GL_n$, our notion of Hodge-Newton reducibility is equivalent to the notion considered by Katz. More precisely, if $\BT$ is a $p$-divisible group over $\overline{\F}_p$ that arises from the datum $(G, [b], \{\mu\})$ with a choice of $b \in [b] \cap G(\completemaxunramint)\mu(p)G(\completemaxunramint)$, the datum $(G, [b], \{\mu\})$ is Hodge-Newton reducible (with respect to some parabolic subgroup and its Levi factor) if and only if the Hodge polygon of $\BT$ passes through a break point of the Newton polygon of $\BT$. See  \cite{Rapoport-Viehmann14}, Remark 4.25 for more details.

We want to interpret the statement of Conjecture \ref{HV conjecture general form} under our assumption. Let us choose $\mu \in \{\mu\}$ and $b \in [b] \cap \levi(\completemaxunram)$ with the properties \ref{hodge-newton type mu factor condition}, \ref{hodge-newton type levi condition}, and \ref{hodge-newton type cocharacters condition} above.  After replacing by a $\sigma$-conjugate if necessary, we may assume that $b \in \levi(\completemaxunramint)\mu(p)\levi(\completemaxunramint)$. Then $b$ and $\levi$ clearly satisfy the condition \ref{HV conjecture nonemptiness condition} of Conjecture \ref{HV conjecture general form}. One can also check that $b$ and $\levi$ satisfy the condition \ref{HV conjecture Jb condition} of Conjecture \ref{HV conjecture general form} (see \cite{Rapoport-Viehmann14}, Remark 8.9.). Moreover, under our assumption the set $I_{b, \{\mu\}, \levi}$ consists of a single element, namely $\{\mu\}_\levi$ (see \cite{Rapoport-Viehmann14}, Theorem 8.8.).

Hence we may state our main theorem as follows:

\begin{theorem}\label{HV conjecture}
Assume that $(G, [b], \{\mu\})$ is Hodge-Newton reducible with respect to a standard parabolic subgroup $\parabolic$ with Levi factor $\levi$. Choose $\mu \in \{\mu\}$ and  $b \in \levi(\completemaxunramint)\mu(p)\levi(\completemaxunramint)$ with the properties \ref{hodge-newton type mu factor condition}, \ref{hodge-newton type levi condition} and \ref{hodge-newton type cocharacters condition} of \ref{definition of hodge-newton reducibility}. Then for any admissible $\overline{\Q}_l$-representation $\rho$ of $J(\Q_p)$, we have an equality of virtual representations of $G(\Q_p) \times \weilgroup{\localreflexfield}$
\[H^\bullet(\RZ^\infty_{G, b})_\rho = \text{Ind}_{\parabolic(\Q_p)}^{G(\Q_p)} H^\bullet(\RZ^\infty_{\levi, b})_\rho.\]
In particular, the virtual representation $H^\bullet(\RZ^\infty_{G, b})_\rho$ contains no supercuspidal representations of $G(\Q_p)$.
\end{theorem}


\subsection{Rigid analytic tower associated to the parabolic subgroup}\label{Rigid analytic tower associated to the parabolic subgroup}$ $

For our proof of Theorem \ref{HV conjecture}, we construct an intermediate tower of rigid analytic spaces associated to the parabolic subgroup $\parabolic$.

\subsubsection{}\label{notations for EL realization}

For the rest of this section, we will always keep the assumption and the notations in the statement of Theorem \ref{HV conjecture}. In addition, we write $\generalhodge{\BT} = (\BT, (t_i))$ for the $p$-divisible group with $G$-structure that arises from the datum $(G, [b], \{\mu\})$ with the choice $b \in [b]$.

By \cite{Hong16}, Lemma 3.1.4,  we can choose a group $\EL{G}$ of EL type with the following properties:
\begin{enumerate}[label=(\roman*)]
\item the embedding $G \hookrightarrow \GL(\module)$ factors through $\EL{G}$,
\item the datum $(\EL{G}, [b], \{\mu\})$ is Hodge-Newton reducible with respect to a proper parabolic subgroup $\EL{\parabolic}$ of $\EL{G}$ and its Levi factor $\EL{\levi}$ such that $\parabolic = \EL{\parabolic} \cap G$ and $\levi = \EL{\levi} \cap G$. 
\end{enumerate}
In general, the group $\EL{G}$ is of the form
\[\EL{G}= \Res_{\integerring_1|\Z_p} \GL_{n_1} \times \Res_{\integerring_2|\Z_p} \GL_{n_2} \times \Res_{\integerring_f|\Z_p} \GL_{n_f}\]
where each $\integerring_j$ is the ring of integers for some finite unramified extension of $\Q_p$. However, in light of functorial properties in 
Proposition \ref{functoriality of rapoport-zink spaces}, we may assume for simplicity that
\[\EL{G}= \ELgroupintegral\] 
where $\integerring$ is the integer ring of some finite unramified extension of $\Q_p$. Then the Levi subgroup $\EL{\levi}$ takes the form 
\begin{equation}\label{EL levi decomp}\EL{\levi} = \ELgroupintegral[m_1] \times \ELgroupintegral[m_2] \times \cdots \times \ELgroupintegral[m_r]. \end{equation}
For each $j=1, 2, \cdots, r$, we define the following data:
\begin{itemize}
\item $\EL{\levi}_j$ is the $j$-th factor in the decomposition \eqref{EL levi decomp},
\item $\levi_j$ is the image of $\levi$ under the projection $\EL{\levi} \twoheadrightarrow \EL{\levi}_j$,
\item $b_j$ is the image of $b$ under the projection $\levi \twoheadrightarrow \levi_j$,
\item $\mu_j$ is the cocharacter of $\levi_j$ induced from $\mu$ via the projection $\levi \twoheadrightarrow \levi_j$.
\end{itemize}

\subsubsection{}\label{Hodge-Newton decomposition and filtration}

For any $x \in \RZ_{G, b}(\overline{\F}_p)$, we write $(\BT_x, (t_{x, i}), \genqisog_x)$ for the corresponding tuple under the moduli description of $\RZ_{G, b}(\overline{\F}_p)$ described in \ref{closed points of RZ functor}, and $\generalhodge{\BT}_x := (\BT_x, (t_{x, i}))$ for the associated $p$-divisible group with $G$-structure. Then we have the following facts from \cite{Hong16}, \S3.2:
\begin{enumerate}[label=(\arabic*)]
\item\label{levi local shimura data} The tuples ($\EL{\levi}_j, [b_j], \{\mu_j\})$ and $(\levi_j, [b_j], \{\mu_j\})$ are unramified local Shimura data of Hodge type for each $j=1, 2, \cdots, s$.
\item\label{Hodge-Newton decomp affine Deligne-Lusztig sets} There is a natural map of the affine Deligne-Lusztig sets
\[\DeligneLusztig{G}{\mu}{b} \stackrel{\sim}{\longrightarrow} \DeligneLusztig{\levi}{\mu}{b} \hooklongrightarrow \DeligneLusztig{\levi_1}{ \mu_1}{ b_1}  \times \cdots \times \DeligneLusztig{\levi_r}{ \mu_r}{ b_r},\]
which induces a natural map
\[ \RZ_{G, b}(\overline{\F}_p) \stackrel{\sim}{\longrightarrow} \RZ_{\levi, b}(\overline{\F}_p) \hooklongrightarrow \RZ_{\levi_1, b_1}(\overline{\F}_p) \times \cdots \times \RZ_{\levi_r, b_r}(\overline{\F}_p)\]
via the natural bijections between the affine Deligne-Lusztig sets and the set of $\overline{\F}_p$-valued points of the Rapoport-Zink spaces (see \ref{closed points of RZ spaces and affine Deligne Lusztig set}.). 
\item\label{Hodge-Newton decomp} The second map in \ref{Hodge-Newton decomp affine Deligne-Lusztig sets} induces a decomposition
\begin{equation*}\generalhodge{\BT}_x = \generalhodge{\BT}_{x_1} \times  \generalhodge{\BT}_{x_2} \times \cdots \times \generalhodge{\BT}_{x_r}\end{equation*}
where $\generalhodge{\BT}_{x_j}$ is the $p$-divisible group with $\levi_j$-structure corresponding to the image of $x$ in $\RZ_{\levi_j, b_j}(\overline{\F}_p)$. 
\item\label{Hodge-Newton filtration} If we set $\BT_x^{(j)} := \BT_{x_j} \times \BT_{x_{j+1}} \times \cdots \times \BT_{x_r}$ for each $j=1, 2, \cdots, r$, the decomposition in \ref{Hodge-Newton decomp} induces a filtration
\[0 \subset \BT_x^{(r)} \subset \BT_x^{(r-1)} \subset \cdots \subset \BT_x^{(1)} = \BT_x\]
such that each quotient $\BT_x^{(j)}/\BT_x^{(j+1)} \simeq \BT_{x_j}$ carries $\levi_j$-structure that arises from the datum $(\levi_j, [b_j], \{\mu_j\})$.
\item\label{lifting the Hodge-Newton filtration} Given a formally smooth $\completemaxunramint$-algebra of the form $\genring = \completemaxunramint[[u_1, \cdots, u_N]]$ or $\genring = \completemaxunramint[[u_1, \cdots, u_N]]/(p^m)$ and a deformation $\deform{\generalhodge{\BT}}_x = (\deform{\BT}_x, (\liftlow{t}_i)) \in \Def_{\BT_x, G}(\genring)$ with an isomorphism $\alpha: \deform{\BT}_x \otimes_\genring \overline{\F}_p \cong \BT_x$, there exists a unique filtration of $\deform{\BT}$
\[ 0 \subset \deform{\BT}_x^{(r)} \subset \deform{\BT}_x^{(r-1)} \subset \cdots \subset \deform{\BT}_x^{(1)} = \deform{\BT}_x \]
which lifts the filtration in \ref{Hodge-Newton filtration} in the following sense: for each $j=1, 2, \cdots, r$, the isomorphism $\alpha$ induces isomorphisms $\alpha^{(j)}: \deform{\BT}_x^{(j)} \otimes_\genring \overline{\F}_p \cong \BT_x^{(j)}$ and $\alpha_j: ({\deform{\BT}_x^{(j)}/\deform{\BT}_x^{(j+1)}}) \otimes_\genring \overline{\F}_p \cong {\BT}_{x_j}$ such that $(\deform{\BT}_x^{(j)}/\deform{\BT}_x^{(j+1)}, (\liftlow{t}_i^{(j)})) \in \Def_{\BT_{x_j}, \levi_j}(\genring)$ for some family of tensors $(\liftlow{t}_i^{(j)})$ on $\D(\deform{\BT}_x^{(j)}/\deform{\BT}_x^{(j+1)})$. 
\end{enumerate}
We refer to the decomposition in \ref{Hodge-Newton decomp} and the filtration in \ref{Hodge-Newton filtration} respectively as the \emph{Hodge-Newton decomposition} and the \emph{Hodge-Newton filtration} of $\generalhodge{\BT}_x$ (with respect to $\parabolic$ and $\levi$). If we take $x$ such that $\BT_x = \BT$, $(t_{x, i}) = (t_i)$ and $\genqisog_x$ is the identity map on $\BT$, we obtain the Hodge-Newton decomposition of $\generalhodge{\BT}$ 
\begin{equation}\label{Hodge-Newton decomp p-divisible group}\generalhodge{\BT} = \generalhodge{\BT}_{1} \times  \generalhodge{\BT}_{2} \times \cdots \times \generalhodge{\BT}_{r}\end{equation}
and the corresponding Hodge-Newton filtration of $\generalhodge{\BT}$ 
\begin{equation}\label{Hodge-Newton filtration p-divisible group} 0 \subset \BT^{(r)} \subset \BT^{(r-1)} \subset \cdots \subset \BT^{(1)} = \BT\end{equation}
where each quotient $\BT^{(j)}/\BT^{(j+1)} \simeq \BT_{j}$ carries $\levi_j$-structure that arises from the datum $(\levi_j, [b_j], \{\mu_j\})$ with the choice $b_j \in [b_j]$.

\subsubsection{} \label{RZ space associated to EL parabolic}


Following Mantovan in \cite{Mantovan08}, Definition 9, we define a set-valued functor $\RZ_{\EL{\parabolic}, b}$ on $\Nilp_{\completemaxunramint}$ as follows: for any $\genring \in \Nilp_{\completemaxunramint}$, we set $\RZ_{\EL{\parabolic}, b}(\genring)$ to be the set of isomorphism classes of triples $(\lift{\BT}, \lift{\BT}^\bullet, \genqisog)$ where
\begin{itemize}
\item $\lift{\BT}$ is a $p$-divisible group over $\genring$ with an action of $\integerring$ (see Example \ref{rapoport-zink spaces of EL type});
\item $\lift{\BT}^\bullet$ is a filtration of $p$-divisible groups over $\genring$
\[0 \subset \lift{\BT}^{(r)} \subset \lift{\BT}^{(r-1)}\subset \cdots \subset \lift{\BT}^{(1)} = \lift{\BT}\] 
which is preserved by the action of $\integerring$ such that the quotients $\lift{\BT}^{(j)}/\lift{\BT}^{(j+1)}$ are $p$-divisible groups (with the induced action of $\integerring$);
\item $\genqisog: \BT_{\genring/p} \to \lift{\BT}_{\genring/p}$ is a quasi-isogeny which is compatible with the action of $\integerring$ and induces quasi-isogenies $\genqisog^{(j)}: \BT^{(j)}_{\genring/p} \longrightarrow \lift{\BT}^{(j)}_{\genring/p}$ for $j=1, 2, \cdots, r$,
\end{itemize}
such that for all $ a \in \integerring$ and $j=1, 2, \cdots, r$,
\[ \det {}_\genring(a, \Lie(\lift{\BT}^{(j)})) = \det(a, \Fil^0(\BT^{(j)})_{\completemaxunram}).\]
Mantovan in \cite{Mantovan08}, Proposition 11 proved that the functor $\RZ_{\EL{\parabolic}, b}$ is represented by a formal scheme which is formally smooth and locally formally of finite type over $\completemaxunramint$. We write $\RZ_{\EL{\parabolic}, b}$ also for this representing formal scheme, and $\RZ_{\EL{\parabolic}, b}^\rig$ for its rigid analytic generic fiber. In addition, we write $\lift{\BT}_{\EL{\parabolic}, b}$ and $\lift{\BT}^\bullet_{\EL{\parabolic}, b}$ respectively for the universal filtered $p$-divisible group over $\RZ_{\EL{\parabolic}, b}$ and the associated ``universal filtration''.

\begin{remark} 
As in \cite{Mantovan08}, Definition 10, we can also define a tower of \'etale covers $\analtower{\RZ_{\EL{\parabolic},b}}=\{\RZ_{\BT, \EL{\parabolic}}^{\EL{\levelatp}'}\}$ over $\RZ_{\EL{\parabolic}, b}^\rig$ with a natural action of $\EL{\parabolic}(\Q_p) \times J_b(\Q_p)$ and a Weil descent datum over $\localreflexfield$, where $\EL{\levelatp}'$ runs over open and compact subgroups of $\EL{\parabolic}(\Z_p)$.
\end{remark}

\subsubsection{}\label{RZ space associated to parabolic}

 By the functoriality of Rapoport-Zink spaces described in Proposition \ref{functoriality of rapoport-zink spaces}, the embedding $G \hooklongrightarrow \EL{G}$ induces a closed embedding 
\[ \RZ_{G, b} \hooklongrightarrow \RZ_{\EL{G}, b}.\]
In addition, we have a natural map 
\[ \EL{\pi}_2: \RZ_{\EL{\parabolic}, b} \longrightarrow \RZ_{\EL{G}, b}\]
defined by $(\lift{\BT}, \lift{\BT}^\bullet, \genqisog)  \mapsto (\lift{\BT}, \genqisog)$ on the points. We define $\RZ_{\parabolic, b} := \RZ_{\EL{\parabolic}, b} \times_{\RZ_{\EL{G}, b}} \RZ_{G, b}$. Then we have the following Cartesian diagram:
\begin{center}
\begin{tikzpicture}[description/.style={fill=white,inner sep=2pt}]
\matrix (m) [matrix of math nodes, row sep=4em,
column sep=3em, text height=1.5ex, text depth=0.25ex]
{ \RZ_{\parabolic, b} & \RZ_{G, b} \\
\RZ_{\EL{\parabolic, b}} & \RZ_{\EL{G}, b}\\
};
\path[-stealth]
    (m-1-1.east) edge node [above] {$\pi_2$} (m-1-2)
    (m-2-1.east) edge node [above] {$\EL{\pi}_2$}(m-2-2);
\draw [right hook-latex] (m-1-1.south) -- (m-2-1);
\draw [right hook-latex] (m-1-2.south) -- (m-2-2);
\end{tikzpicture}
\end{center}
Moreover, $\pi_2$ is a local isomorphism which gives an isomorphism on the rigid analytic generic fiber since $\EL{\pi}_2$ has the same properties (see \cite{Mantovan08}, Theorem 36 and \cite{Shen13}, Proposition 6.3.).


We want to describe the universal property of the closed embedding $\RZ_{\parabolic, b} \hooklongrightarrow \RZ_{\EL{\parabolic}, b}$ in an analogous way to the universal property of $\RZ_{G, b} \subset \RZ_b$ described in \ref{rapoport-zink space of hodge type as a functor}. For this, we choose a decomposition of $\module$
\[ \module = \module_1 \oplus \module_2 \oplus \cdots \oplus \module_r\]
corresponding to the decomposition of $\EL{\levi}$ in \eqref{EL levi decomp}. We set $\module^{(j)}= \module_1 \oplus \cdots \oplus \module_j$ for $j=1, 2, \cdots, r$, and denote by $\module^\bullet$ the filtration
\[ 0 \subset \module^{(1)} \subset \cdots \subset \module^{(r)} = \module.\]
Then for any $\Z_p$-algebra $\genring$ we have
\[\parabolic(\genring) = \{g \in G(\genring): g(\module_\genring^{\bullet}) = \module_\genring^{\bullet}\}.\]

Now consider a morphism $f: \Spf (\genring) \to \RZ_{\EL{\parabolic}, b}$ for some $\genring \in \Nilp_{\completemaxunramint}$. Let $(\lift{\BT}, \lift{\BT}^\bullet)$ be a $p$-divisible group over $\Spec (\genring)$ with a filtration which pulls back to $(f^*\lift{\BT}_{\EL{\parabolic}, b}, f^*\lift{\BT}_{\EL{\parabolic}, b}^\bullet)$ over $\Spf(\genring)$. We denote by $\D(\lift{\BT}^\bullet)$ the filtration of Dieudonn\'e modules
\[ 0 = \D(\lift{\BT}/\lift{\BT}^{(1)})  \subset \D(\lift{\BT}/\lift{\BT}^{(2)}) \subset \cdots \subset \D(\lift{\BT}/\lift{\BT}^{(r)}) \subset  \D(\lift{\BT})\]
induced by $\lift{\BT}^\bullet$ via (contravariant) Dieudonne theory. We choose tensors $(\liftother{t}_i)$ on $\D(\lift{\BT})[1/p]$ 
as in \ref{rapoport-zink space of hodge type as a functor}. Then $f$ factors through $\RZ_{\parabolic, b}$ if and only if $\EL{\pi}_2 \circ f$ factors through $\RZ_{G, b} \hooklongrightarrow \RZ_{\EL{G},b}$, which is equivalent to existence of a (unique) family of tensors $(\liftlow{t}_i)$ on $\D(\lift{\BT})$ 
such that
\begin{enumerate}[label=(\roman*)]
\item\label{parabolic RZ space isogeny condition} for some ideal of definition $J$ of $\genring$ containing $p$, the pull-back of $(\liftlow{t}_i)$ over $\genring/J$ agrees with the pull-back of $(\liftother{t}_i)$ over $\genring/J$,
\item\label{parabolic RZ space parabolic action condition} for a $p$-adic lift $\lift{\genring}$ of $\genring$ which is formally smooth over $\completemaxunramint$, the $\lift{\genring}$-scheme
\[ \isommatchingtensors_{\lift{\genring}} := \Isom_{\lift{\genring}}\Big( [\D(\lift{\BT})_{\lift{\genring}}, (\liftlow{t}_i)_{\lift{\genring}}], [ \module^* \otimes_{\Z_p} \lift{\genring} , (s_i \otimes 1)]\Big)\]
defined in \ref{rapoport-zink space of hodge type as a functor} is a $G$-torsor, and consequently the $\lift{\genring}$-scheme
\[ \isommatchingtensors'_{\lift{\genring}} := \Isom_{\lift{\genring}}\Big( [\D(\lift{\BT^\bullet})_{\lift{\genring}}, (\liftlow{t}_{i})_{\lift{\genring}}], [ (\module^{\bullet})^* \otimes_{\Z_p} \lift{\genring} , (s_i \otimes 1)]\Big)\]
is a $\parabolic$-torsor,
\item\label{parabolic RZ space Hodge filtration condition} the Hodge filtration of $\lift{\BT}$ 
is a $\{\mu\}$-filtration with respect to $(\liftlow{t}_i)$.
\end{enumerate}
Here the scheme $\isommatchingtensors'_{\lift{\genring}}$ in \ref{parabolic RZ space parabolic action condition} classifies the isomorphisms $\D(\lift{\BT})_{\lift{\genring}} \cong \module^*_{\lift{\genring}}$ which map the tensors $(\liftlow{t}_i)$ to $(s_i \otimes 1)$ and the filtration $\D(\lift{\BT^\bullet})_\genring$ to $(\module^{\bullet})^* \otimes_{\Z_p} \lift{\genring}$. 

We obtain the ``universal $p$-divisible group" $\lift{\BT}_{\parabolic, b}$ over $\RZ_{\parabolic, b}$ with the associated 	``universal filtration'' $\lift{\BT}_{\parabolic, b}^\bullet$ by taking the pull-back of $\lift{\BT}_{\EL{\parabolic}, b}$ and $\lift{\BT}_{\EL{\parabolic}, b}^\bullet$ over $\RZ_{\parabolic, b}$. We also obtain a family of ``universal tensors" $(\liftlow{t}^{\univ, \parabolic}_i)$ on $\D(\lift{\BT}_{\parabolic, b})$ by applying the universal property to an open affine covering of $\RZ_{\parabolic, b}$. Moreover, this family has a ``\'etale realization'' $(\liftlow{t}^{\univ, \parabolic}_{i, \et})$ on the Tate module $\tatemodule{\lift{\BT}_{\parabolic, b}}$ (see \cite{Kim13}, Theorem 7.1.6.).

\subsubsection{}\label{cohomology of parabolic RZ space}

The formal scheme  $\RZ_{\parabolic, b}$ is formally smooth and locally formally of finite type over $\completemaxunramint$ by construction. Hence it admits a rigid analytic generic fiber which we denote by $\RZ_{\EL{\parabolic}, b}^\rig$. Moreover, since $\pi_2$ gives an isomorphism on the rigid analytic generic fiber, we have a $J_b(\Q_p)$-action and a Weil descent datum over $\localreflexfield$ on $\RZ_{\parabolic, b}^\rig$ induced by the corresponding structures on $\RZ_{G, b}^\rig$.

For any open compact subgroup $\levelatp'$ of $\parabolic(\Z_p)$, we define the following rigid analytic \'etale cover of $\RZ_{\parabolic, b}^\rig$:
\[\RZ_{\parabolic, b}^{\levelatp'} := \Isom_{\RZ_{\parabolic, b}^\rig}\Big( [\module^\bullet, (s_i)],  [ \tatemodule{\lift{\BT}^\bullet_{\parabolic, b}} , (\liftlow{t}^{\univ, \parabolic}_{i, \et}) ] \Big)\Big/\levelatp'.\]
The $J_b(\Q_p)$-action and the Weil descent datum over $\localreflexfield$ on $\RZ_{\parabolic, b}^\rig$ pull back to $\RZ_{\parabolic, b}^{\levelatp'}$.  We denote by $\analtower{\RZ_{\parabolic, b}}: = \{\RZ_{\parabolic, b}^{\levelatp'}\}$ the tower of these covers with Galois group $\parabolic(\Z_p)$. The Galois action on this tower gives rise to a natural $\parabolic(\Q_p)$-action which commutes with the $J_b(\Q_p)$-action and the Weil descent datum over $\localreflexfield$ (cf. \cite{Kim13}, Proposition 7.4.8.). Hence the cohomology groups 
\[ H^i(\RZ_{\parabolic, b}^{\levelatp'}) = H^i_c (\RZ_{\parabolic, b}^\levelatp \otimes_{\completemaxunram} \C_p, \Q_l( \dim \RZ_{\parabolic, b}^{\levelatp'}))\]
form a tower $\{ H^i(\RZ_{\parabolic, b}^{\levelatp'})\}$ for each $i$, which are endowed with a natural action of $\parabolic(\Q_p) \times \weilgroup{\localreflexfield} \times J_b(\Q_p)$. Moreover, for any admissible $l$-adic representation $\rho$ of $J_b(\Q_p)$, the groups
\[H^{i, j}(\analtower{\RZ_{\parabolic, b}})_\rho := \varinjlim_{\levelatp'} \Ext^j_{J_b(\Q_p)} (H^i(\RZ_{\parabolic, b}^{\levelatp'}), \rho)\]
satisfy the following properties (cf. \ref{cohomology of rigid analytic tower}):
\begin{enumerate}
\item The groups $H^{i, j}(\analtower{\RZ_{\parabolic, b}})_\rho$ vanish for almost all $i, j$. 
\item There is a natural action of $\parabolic(\Q_p) \times \weilgroup{\localreflexfield}$ on each $H^{i, j}(\analtower{\RZ_{ \parabolic, b}})_\rho$. 
\item The representations $H^{i, j}(\analtower{\RZ_{\parabolic, b}})_\rho$ are admissible. 
\end{enumerate}
We can thus define a virtual representation of $\parabolic(\Q_p) \times \weilgroup{\localreflexfield}$
\[ H^\bullet(\RZ^\infty_{\parabolic,b})_\rho := \sum_{i, j \geq 0} (-1)^{i+j} H^{i, j} (\analtower{\RZ_{\parabolic, b}})_\rho.\]

\begin{remark}
Alternatively, we can obtain the tower $\analtower{\RZ_{\parabolic, b}}$ as the pull-back of the tower $\analtower{\RZ_{\EL{\parabolic}, b}}$ over $\RZ_{\parabolic,b}^\rig$. 
\end{remark}

\subsection{Harris-Viehmann conjecture: proof}$ $

We finally present our proof of Theorem \ref{HV conjecture}. We retain all the notations from \ref{Rigid analytic tower associated to the parabolic subgroup}. 

\begin{lemma}\label{diagram of three towers} 
There exists a diagram
\begin{center}
\begin{tikzpicture}[description/.style={fill=white,inner sep=2pt}]
\matrix (m) [matrix of math nodes, row sep=3em,
column sep=1.5em, text height=1.5ex, text depth=0.25ex]
{  & \RZ_{\parabolic, b}^\rig & \\
\RZ_{\levi, b}^\rig& & \RZ_{G, b}^\rig\\
};
\path[-stealth]
    (m-1-2) edge node [below] {\hspace{.8em}$\pi_1$} (m-2-1)
    (m-2-1) edge[bend left] node [above] {$s$} (m-1-2)
    (m-1-2) edge node [below] {$\pi_2$\hspace{.8em}$ $} (m-2-3);
\end{tikzpicture}
\end{center}
such that
\begin{enumerate}[label=(\arabic*)]
\item\label{key lemma s} $s$ is a closed immersion, 
\item\label{key lemma pi_1} $\pi_1$ is a fibration in balls, 
\item\label{key lemma pi_2} $\pi_2$ is an isomorphism. 
\end{enumerate}
\end{lemma}

\begin{proof}
For notational simplicity, we assume that $r=2$, i.e., the decomposition of $\EL{\levi}$ in \eqref{EL levi decomp} has two factors. Our argument will naturally extend to the general case.

Note that we have already constructed $\pi_2$ and proved \ref{key lemma pi_2} in \ref{RZ space associated to parabolic}.

Let us now prove \ref{key lemma s}. From the decomposition $\EL{\levi} = \EL{\levi}_1 \times \EL{\levi}_2$ we obtain a natural isomorphism $\RZ_{\EL{\levi}, b} \simeq \RZ_{\EL{\levi}_1, b_1} \times \RZ_{\EL{\levi}_2, b_2}$ by Proposition \ref{functoriality of rapoport-zink spaces}. Consider the map 
\[\EL{s}: \RZ_{\EL{\levi}, b} \simeq \RZ_{\EL{\levi}_1, b_1} \times \RZ_{\EL{\levi}_2, b_2} \longrightarrow \RZ_{\EL{\parabolic}, b}\]
where the second arrow is defined by $(\lift{\BT}_1, \genqisog_1, \lift{\BT}_2, \genqisog_2) \mapsto (\lift{\BT}_1 \times \lift{\BT}_2, 0 \subset \lift{\BT}_2 \subset \lift{\BT}_1 \times \lift{\BT}_2, \genqisog_1 \times \genqisog_2)$ on the points. Then $\EL{s}$ gives a closed immersion on the rigid analytic generic fibers by \cite{Mantovan08}, Proposition 14. We define $s$ to be the restriction of $\EL{s}$ on $\RZ_{\levi, b}$. Since $s$ also gives a closed immersion on the rigid analytic generic fibers by construction, it suffices to show that $s$ factors through the embedding $\RZ_{\parabolic, b} \hooklongrightarrow \RZ_{\EL{\parabolic}, b}$, which amounts to proving that $\EL{\pi}_2 \circ s$ factors through $\RZ_{G, b}$. In fact, $\EL{\pi}_2 \circ s$ is the natural closed embedding $\RZ_{\levi, b} \hooklongrightarrow \RZ_{\EL{G}, b}$ which is functorially induced by the embedding $\levi \hooklongrightarrow \EL{G}$ in the sense of Proposition \ref{functoriality of rapoport-zink spaces}. Hence $\EL{\pi}_2 \circ s$ factors through $\RZ_{G, b}$ as the embedding $\levi \hooklongrightarrow \EL{G}$ factors through $G$.

It remains to prove \ref{key lemma pi_1}. Note that we have a natural embedding 
\[\RZ_{\levi, b} \hooklongrightarrow \RZ_{\EL{\levi}, b}\] 
which is functorially induced by the embedding $\levi \hooklongrightarrow \EL{\levi}$ in the sense of Proposition \ref{functoriality of rapoport-zink spaces}. Consider the map 
\[\EL{\pi}_1: \RZ_{\EL{\parabolic}, b} \longrightarrow  \RZ_{\EL{\levi}_1, b_1} \times \RZ_{\EL{\levi}_2, b_2} \stackrel{\sim}{\longrightarrow} \RZ_{\EL{\levi}, b}\]
defined by $(\lift{\BT}, \lift{\BT}^\bullet, \genqisog) \mapsto (\lift{\BT}/\lift{\BT}^{(2)}, \genqisog/\genqisog^{(2)}, \lift{\BT}^{(2)}, \genqisog^{(2)}) \mapsto ((\lift{\BT}/\lift{\BT}^{(2)}) \times \lift{\BT}^{(2)}, (\genqisog/\genqisog^{(2)}) \times \genqisog^{(2)})$ on the points, where $\genqisog/\genqisog^{(2)}: (\BT_1)_{\genring/p} = (\BT/\BT^{(2)})_{\genring/p} \longrightarrow (\lift{\BT}/\lift{\BT}^{(2)})_{\genring/p}$ is a quasi-isogeny induced by $\genqisog$ and $\genqisog^{(2)}$. We define $\pi_1$ be the restriction of $\EL{\pi}_1$ on $\RZ_{\parabolic, b}$.

We claim that $\pi_1$ factor through the embedding $\RZ_{\levi, b} \hooklongrightarrow \RZ_{\EL{\levi}, b}$. It suffices to show that (locally) the map $\pi_2\inv \circ \pi_1$ factors through $\RZ_{\levi, b} \hooklongrightarrow \RZ_{\EL{\levi}, b}$. We only need to check this on the set of $\overline{\F}_p$-points and the completions thereof. On the set of $\overline{\F}_p$-points, $\pi_2\inv \circ \pi_1$ coincides with the map in \ref{Hodge-Newton decomp affine Deligne-Lusztig sets} of \ref{Hodge-Newton decomposition and filtration} and thus factors through $\RZ_{\levi, b}$. On the completion $\completion{(\RZ_{G, b})_x}$ at $x \in \RZ_{\BT, G}(\overline{\F}_p)$, we get a map 
\[\Def_{\BT_x, G} \longrightarrow \Def_{\BT_{x_1}, \EL{\levi}_1} \times \Def_{\BT_{x_2}, \EL{\levi}_2} \simeq \Def_{\BT_{x_1} \times \BT_{x_2}, \EL{\levi}_1 \times \EL{\levi}_2} = \Def_{\BT_x, \EL{\levi}}\]
induced by the association $\deform{\BT}_x \mapsto (\deform{\BT}_x/\deform{\BT}_x^{(2)}) \times \deform{\BT}_x^{(2)}$. Note that $(\deform{\BT}_x/\deform{\BT}_x^{(2)}) \times \deform{\BT}_x^{(2)}$ is a deformation of $\BT_x$ via the isomorphism $\alpha_1 \times \alpha^{(2)}$ in \ref{lifting the Hodge-Newton filtration} of \ref{Hodge-Newton decomposition and filtration}. Since this isomorphism is induced by $\alpha$, we see that $(\deform{\BT}_x/\deform{\BT}_x^{(2)}) \times \deform{\BT}_x^{(2)}$ lifts the tensors that define $G$-structure on $\BT_x$. Hence the image of the above map must lie in $\Def_{\BT_x, \EL{\levi}} \cap \Def_{\BT_x, G} = \Def_{\BT_x, \levi}$.

Finally, we easily see that $\pi_1$ is a fibration in balls. In fact, for any point $x \in \RZ_{\levi, b}(\overline{\F}_p)$ the completion of $\RZ_{\parabolic, b}$ at $s(x)$ is isomorphic to a formal deformation space of $\BT_x$ with Tate tensors, which is isomorphic to a formal spectrum of  a power series ring over $\completemaxunramint$ as proved in \cite{Faltings99}, \S7 (see also \cite{Moonen98}, \S4.). 
\end{proof}

\begin{prop} \label{comparing cohomologies of levi and parabolic tower} For any admissible $l$-adic representation $\rho$ of $J_b(\Q_p)$, we have
\[ H^\bullet(\RZ^\infty_{\levi, b})_\rho = H^\bullet(\RZ^\infty_{\parabolic, b})_\rho \]
as virtual representations of $\parabolic(\Q_P) \times \weilgroup{\localreflexfield}$. 
\end{prop}

\begin{proof}
For any open compact subgroups $\levelatp' \subseteq \parabolic(\Z_p)$, we get morphisms of rigid analytic spaces
\[ s_{\levelatp'} : \RZ_{\levi, b}^{\levelatp' \cap \levi(\Q_p)} \longrightarrow \RZ_{\parabolic, b}^{\levelatp'} \quad \text{ and } \quad \pi_{1, \levelatp'}: \RZ_{\parabolic, b}^{\levelatp'} \longrightarrow \RZ_{\levi, b}^{\levelatp' \cap \levi(\Q_p)}\]
which are $\parabolic(\Q_p) \times J_b(\Q_p)$-equivariant and compatible with the Weil descent datum. Moreover, $s_{\levelatp'}$'s are closed immersions and satisfy $\pi_{1, \levelatp'} \circ s_{\levelatp'} = \text{id}_{\RZ_{\levi, b}^{\levelatp' \cap \levi(\Q_p)}}$.

Recall that we have a universal $p$-divisible group $\lift{\BT}_{\EL{\parabolic}, b}$ over $\RZ_{\BT, \EL{\parabolic}}$ with the associated filtration $\lift{\BT}^\bullet_{\EL{\parabolic}, b}$. By \cite{Mantovan08}, Proposition 30, we have a formal scheme $\RZ_{\EL{\parabolic}, b}^{(m)} \longrightarrow \RZ_{\EL{\parabolic}, b}$ for each integer $m>0$ with the following properties:
\begin{enumerate}[label=(\roman*)]
\item a morphism $f: \Spf(\genring) \longrightarrow \RZ_{\EL{\parabolic}, b}$ for some $\genring \in \Nilp_{\completemaxunramint}$ factors through $\RZ_{\EL{\parabolic}, b}^{(m)}$ if and only if the filtration $f^*\lift{\BT}^\bullet_{\EL{\parabolic}, b}[p^m]$ is split, 
\item the formal schemes $\RZ_{\EL{\parabolic}, b}^{(m)}$ and $\RZ_{\EL{\parabolic}, b}$ become isomorphic when considered as formal schemes over $\RZ_{\EL{\levi}, b}$ via the map $\EL{\pi}_1: \RZ_{\EL{\parabolic}, b} \longrightarrow \RZ_{\EL{\levi}, b}$. 
\end{enumerate}
Taking the pull back of $\RZ_{\EL{\parabolic}, b}^{(m)}$ over $\RZ_{\parabolic, b}$, we obtain a formal scheme $\RZ_{\parabolic, b}^{(m)} \longrightarrow \RZ_{\parabolic, b}$ for each integer $m>0$ with analogous properties. We write $\RZ_{\parabolic, b}^{(m), \rig}$ for the rigid analytic generic fiber of $\RZ_{\parabolic, b}^{(m)}$.

For each integer $m>0$, we set $\levelatp'^{(m)}:= \ker\big(\parabolic(\Z_p) \twoheadrightarrow \parabolic(\Z_p/p^m\Z_p)\big)$ and define two distinct covers
$\mathcal{P}_m \longrightarrow \RZ_{\parabolic, b}^{(m)}$ and $\mathcal{P}'_m \longrightarrow \RZ_{\parabolic, b}^{(m)}$ by the following Cartesian diagrams:
\begin{center}
\begin{tikzpicture}[description/.style={fill=white,inner sep=4pt}]
\matrix (m) [matrix of math nodes, row sep=4em,
column sep=3em, text height=1.5ex, text depth=0.7ex]
{ \mathcal{P}_m & \RZ_{\parabolic, b}^{(m), \rig} \\
\RZ_{\parabolic, b}^{\levelatp'^{(m)}} & \RZ_{\parabolic, b}^\rig\\
};
\path[-stealth]
    (m-1-1.east) edge (m-1-2)
    (m-2-1.east) edge (m-2-2)
    (m-1-1.south) edge (m-2-1)
    (m-1-2.south) edge (m-2-2);
\end{tikzpicture}
\quad
\begin{tikzpicture}[description/.style={fill=white,inner sep=4pt}]
\matrix (m) [matrix of math nodes, row sep=4em,
column sep=3em, text height=2.3ex, text depth=0.25ex]
{ \mathcal{P}'_m & \RZ_{\parabolic, b}^{(m), \rig} \\
\RZ_{\levi, b}^{\levelatp'^{(m)}} & \RZ_{\levi, b}^\rig\\
};
\path[-stealth]
    (m-1-1.east) edge (m-1-2)
    (m-2-1.east) edge (m-2-2)
    (m-1-1.south) edge (m-2-1)
    (m-1-2.south) edge node [right] {$\pi_1$} (m-2-2);
\end{tikzpicture}
\end{center}
Since $\pi_1$ is a fibration in balls, we obtain quasi-isomorphisms
\[ R\Gamma_c(\mathcal{P}'_m \otimes_{\completemaxunram} \C_p, \overline{\Q}_l) \cong R\Gamma_c(\RZ_{\levi, b}^{\levelatp'^{(m)}} \otimes_{\completemaxunram} \C_p, \overline{\Q}_l(-D))[-2D]\quad \text{ for all }m>0\]
where $D = \dim \RZ_{\parabolic, b} - \dim \RZ_{\levi, b}$. Moreover, we can argue as in \cite{Mantovan08}, Lemma 31 and Proposition 32 to deduce quasi-isomorphisms 
\[ R\Gamma_c(\RZ_{\parabolic, b}^{\levelatp'^{(m)}} \otimes_{\completemaxunram} \C_p, \overline{\Q}_l) \cong R\Gamma_c(\mathcal{P}'_m \otimes_{\completemaxunram} \C_p, \overline{\Q}_l) \quad \text{ for all }m>0.\]
Thus we have quasi-isomorphisms
\[ R\Gamma_c(\RZ_{\parabolic, b}^{\levelatp'^{(m)}} \otimes_{\completemaxunram} \C_p, \overline{\Q}_l) \cong R\Gamma_c(\RZ_{\levi, b}^{\levelatp'^{(m)}} \otimes_{\completemaxunram} \C_p, \overline{\Q}_l(-D))[-2D] \quad \text{ for all }m >0,\]
which yield the desired equality.  \end{proof}

\begin{prop} \label{comparing cohomologies of parabolic and original tower} For any admissible $l$-adic representation $\rho$ of $J_b(\Q_p)$, we have
\[ H^\bullet(\RZ^\infty_{G, b})_\rho = \text{Ind}_{\parabolic(\Q_p)}^{G(\Q_p)} H^\bullet(\RZ^\infty_{\parabolic, b})_\rho \]
as virtual representations of $\parabolic(\Q_P) \times \weilgroup{\localreflexfield}$. 

\end{prop}

\begin{proof}
For any open compact subgroup $\levelatp \subseteq G(\Z_p)$, we have natural morphisms of rigid analytic spaces
\[\pi_{2, \levelatp}: \RZ_{\parabolic, b}^{\levelatp \cap \parabolic(\Q_p)} \longrightarrow \RZ_{G, b}^\levelatp\]
which are $\parabolic(\Q_p) \times J_b(\Q_p)$-equivariant and compatible with the Weil descent datum. Moreover, these maps are evidently closed immersions. Hence we have isomorphisms
\[ \RZ_{G, b}^\levelatp \cong \RZ_{G, b}^\levelatp \times_{\RZ_{G, b}^\rig} \RZ_{\parabolic, b}^\rig \cong \coprod_{\levelatp \backslash G(\Q_p) / \parabolic(\Q_p) } \RZ_{\parabolic, b}^{\levelatp \cap \parabolic(\Q_p)}\quad \text{for all }\levelatp \subseteq G(\Z_p),\]
thereby obtaining the desired identity. \end{proof}

Proposition \ref{comparing cohomologies of levi and parabolic tower} and \ref{comparing cohomologies of parabolic and original tower} together imply Theorem \ref{HV conjecture}.

\bigskip



 \end{document}